\documentclass[11pt, reqno]{amsart}
\usepackage{amsfonts}
\usepackage{amsmath}
\usepackage{amssymb}
\usepackage{amsthm}
\usepackage{enumerate}
\usepackage{mathdots}
\usepackage{hyperref}
\usepackage{caption}
\usepackage{adjustbox}
\usepackage{bbold}
\usepackage{mathtools}
\usepackage{hyperref}
\newtheorem{theorem}{Theorem}
\newtheorem{lemma}{Lemma}
\newtheorem{proposition}{Proposition}

\newtheorem{cor}{Corollary}

\newtheorem{prob}{Problem}

\newtheorem*{rmk*}{Remark}

\DeclareMathOperator{\sgn}{sgn}
\usepackage{fancyhdr}

\title{On the zeros of reciprocal Littlewood polynomials}
\author{Benjamin Bedert\\
\tiny Mathematical Institute, University of Oxford}
\thanks{benjamin.bedert@maths.ox.ac.uk\\The author gratefully acknowledges financial support from the EPSRC}
\begin{document}
\maketitle
\begin{abstract}
Let $P(z)=\sum_{n=0}^Na_nz^n$ be a Littlewood polynomial of degree $N$, meaning that $a_n\in\{\pm 1\}$. We say that $P$ is reciprocal if $P(z)=z^NP(1/z)$. Borwein, Erd\'elyi and Littmann posed the question of determining the minimum number $Z_{\mathcal{L}}(N)$ of zeros of modulus 1 of a reciprocal Littlewood polynomial $P$ of degree $N$. Several finite lower bounds on $Z_{\mathcal{L}}(N)$ have been obtained in the literature, and it has been conjectured by various authors that $Z_{\mathcal{L}}(N)$ must in fact grow to infinity with $N$. Starting from ideas in recent breakthrough papers of Erd\'elyi and Sahasrabudhe, we are able to confirm this.
\end{abstract}

\tableofcontents
\section{Introduction}
Littlewood in his 1968 monograph “Some Problems in Real and Complex
Analysis” \cite[Problem 22]{littlewoodproblems} posed the following problem:
\begin{prob} \label{littprob} If the $n_j$ are distinct integers, what is the lower bound on the number of real zeros of $\sum^N_{j=1} \cos n_j\theta$ in a period $[0,2\pi]$? Possibly $N-1$, or not much
less.
\end{prob}
No progress was made on this problem until the breakthrough of Borwein, Erd\'elyi, Ferguson and Lockhart \cite{borwein} who demonstrated the existence of cosine polynomials $\sum^N_{j=1} \cos n_j\theta$ with $O(N^{5/6}\log N)$ roots in a period, thus giving a counterexample to Littlewood's proposed lower bound. Their construction was later used independently by Konyagin \cite{konyaginzeros} and Ju\v skevi\v cius and Sahasrabudhe \cite{juskevicius} to obtain the improved upper bound $O((N\log N)^{2/3})$. Obtaining lower bounds for the number of zeros of 
\pagebreak

\begin{flushleft} cosine polynomials of the form
\end{flushleft}
\begin{equation}
    f(\theta)=\sum_{j=1}^N \cos n_j\theta
\label{1cosinepoly}
\end{equation}
with the $n_j$ being distinct non-negative integers seems to be a hard problem; in fact it was only recently proved that the number of such zeros grows to infinity as $N\to \infty$. This was proved independently by Erd\'elyi \cite{erdelyioriginal}, and by Sahasrabudhe \cite{sahasrabudhe} who further obtained the explicit lower bound $(\log\log\log N)^{1/2-o(1)}$ for the number of zeros of \eqref{1cosinepoly}. Erd\'elyi \cite{erdelyi} combined the arguments of \cite{erdelyioriginal} and \cite{sahasrabudhe} to obtain the slightly better exponent $1-o(1)$ replacing $1/2-o(1)$ in Sahasrabudhe's bound. Let us write $Z(f)$ for the number of zeros of the cosine polynomial $f$ in $[0,2\pi]$. Then the current best bounds for Problem \ref{littprob} are \begin{equation}
    (\log \log \log N)^{1-o(1)}\leqslant \min_f Z(f)\leqslant (N\log N)^{2/3},
\label{Zbounds}
\end{equation}
where the minimum is taken over all cosine polynomials of the form \eqref{1cosinepoly} with $N$ terms, and it is an interesting open problem problem to determine the true order of magnitude. 
The methods of \cite{erdelyioriginal,erdelyi,sahasrabudhe} further prove lower bounds on the number of zeros of a more general class of cosine polynomials. Let $R$ be a finite set, let
\begin{equation}
    f(\theta) = \sum_{n=0}^N A_n \cos n\theta
\label{cosinepolygene}
\end{equation}
with coefficients $A_n\in R$ and let $Z_R(N)=\min_f Z(f)$ where the minimum is taken over all cosine polynomials $f$ of degree $N$ with coefficients in $R$, as in \eqref{cosinepolygene}. The following lower bound is established in \cite{erdelyi} (following similar theorems in \cite{erdelyioriginal,sahasrabudhe}).
\begin{theorem}[\cite{erdelyi}, Theorem 2.1] \label{erdetheo}
    Let $R\subset \mathbf{Z}$ be finite and $M(R)=\max_{r\in R} |r|$.
    Let $f$ be a cosine polynomial as in \eqref{cosinepolygene} with coefficients $A_n\in R$. Then the number of roots of $f$  satisfies
    \begin{equation*}
        Z(f)\geqslant \left(\frac{c}{1+\log M(R)}\right)\frac{\log\log\log |f(0)|}{\log\log\log\log |f(0)|}-1,
    \end{equation*}
    where $c>0$ is an absolute constant.
\end{theorem}
Applying this theorem  with $R=\{0,1\}$ clearly yields the lower bound in \eqref{Zbounds}. One should note that for general sets of coefficients $R$, a lower bound as in Theorem \ref{erdetheo} has to depend on a slightly more subtle parameter like $|f(0)|$ rather than simply the degree $N$ of \eqref{cosinepolygene}. This is pointed out in \cite{erdelyioriginal} using the explicit functions
\begin{equation}
    g_N(\theta)= \cos \theta+\sum_{j=0}^N(-1)^j\cos((2j+1)\theta)
\label{fewrootspolys}
\end{equation}
which have no more than 2 roots in $[0,2\pi]$ for all $N$.\footnote{This is compatible with Theorem \ref{erdetheo} as the sequence $|g_N(0)|$ is bounded} Similarly, if the set of coefficients $R$ has both positive and negative numbers, then Theorem \ref{erdetheo} does not show that the number of zeros of \eqref{cosinepolygene} tends to infinity with the degree $N$ and hence it gives no lower bound on $Z_R(N)$. For a general finite set $R$, Borwein and Erd\'elyi \cite{borweinerdelyitrunc} proved that if $(a_n)\in R^\mathbf{N}$ is a fixed sequence with finitely many zero terms, then $\lim_{n\to \infty} Z(f_n)=\infty$ where $f_n(\theta)=\sum_{j=0}^n a_j\cos j\theta$ are truncations of a fixed series. Note however that this result also does not show that $Z_R(N)\to \infty$ as $N\to\infty$. In fact, such a result is not true for general $R$ as can been seen by taking $R=\{0,\pm1,2\}$ corresponding to \eqref{fewrootspolys}.
\smallskip

In this paper we consider the class of cosine polynomials with coefficients in $R=\{\pm1\}$ which is of special interest in the literature, and our goal will be to prove the conjecture that $Z_{\{\pm1\}}(N)\to \infty$ as $N\to\infty$. In this context, problems are usually formulated equivalently using reciprocal Littlewood polynomials.
Following the notation in \cite{erdelyi}, we denote the class of Littlewood polynomials of degree $N$ by
\begin{equation*}
    \mathcal{L}_N=\left\{P: P(z)=\sum_{n=0}^N \varepsilon_n z^n, \varepsilon_n\in\{\pm 1\}\right\}
\end{equation*}
and we say that a polynomial $P(z)=\sum_{n=0}^N a_n z^n$ is \textit{reciprocal} if $P(z)=z^NP(1/z)$, or equivalently $a_j=a_{N-j}$ for all $j$. Let us write $Z(P)=\left|\{\rho\in\mathbf{C}:|\rho|=1, P(\rho)=0\}\right|$ for the number of unimodular roots of $P$, and let
\begin{equation}
    Z_{\mathcal{L}}(N) = \min_{P\in \mathcal{L}_N \text{ reciprocal}} Z(P)
\label{miniLitt}
\end{equation}
be the minimum number of unimodular roots of a reciprocal Littlewood polynomial of degree $N$. Borwein, Erd\'elyi and Littmann \cite[Problem 1]{littmann} asked to determine $Z_{\mathcal{L}}(N)$. It was shown independently by Erd\'elyi \cite{erdelyioneroot} and Mercer \cite{mercer} that $Z_{\mathcal{L}}(N)\geqslant 1$ for $N\geqslant 1$.\footnote{Mercer also showed that `skew-reciprocal' Littlewood polynomials do not have any unimodular roots, justifying that the minimum in \eqref{miniLitt} is taken over reciprocal polynomials only.} Mukunda \cite{mukunda} improved this result by proving that $Z_{\mathcal{L}}(N)\geqslant 3$ for odd degrees $N\geqslant 3$. The current best bounds are due to Drungilas \cite{drungilas} who proved that $Z_{\mathcal{L}}(N)\geqslant 5$ for odd $N\geqslant 7$ and $Z_{\mathcal{L}}(N)\geqslant 4$ for even $N\geqslant 14$. It was also shown in \cite{borwein} that reciprocal polynomials in $\mathcal{L}_N$ have at least $n/4$ roots on average. Erd\'elyi \cite{erdelyi14.18, erdelyioriginal,erdelyi}, Mukunda \cite{mukunda} and Drungilas \cite{drungilas} state the conjecture that \begin{equation}
    \lim_{N\to \infty} Z_{\mathcal{L}}(N) = \infty.
\label{conjecture}    
\end{equation} Our main result confirms this conjecture.

\begin{theorem} \label{maintheorem}
    We have that $Z_{\mathcal{L}}(N)\to \infty$ as $N\to \infty$.
\end{theorem}
The analogous question in the setting of cosine polynomials appeared recently in Sahasrabudhe's paper \cite{sahasrabudhe}. We also answer this, showing that the number of zeros in $[0,2\pi]$ of a cosine polynomial $f(\theta)=\sum_{n=0}^N \varepsilon_n \cos n\theta$ with $\varepsilon_n\in\{-1,1\}$ tends to infinity with the degree $N$. Note again that Theorem \ref{erdetheo} does not apply in this setting since $|f(0)|$ can remain bounded. Rewriting such $f$ as a reciprocal polynomial in $z=e^{i\theta}$ shows that Theorem \ref{cosinetheorem} is essentially equivalent to the claim that $Z_{\mathcal{L}}(2N)\to \infty$ as $N\to \infty$.
\begin{theorem} 
\label{cosinetheorem}
    We have that $Z_{\{\pm1\}}(N)\to\infty$ as $N\to\infty$.
\end{theorem}
\textbf{Acknowledgements}
The author would like to thank Tam\'as Erd\'elyi and his supervisor Ben Green for their useful comments on earlier versions of the paper and for sharing helpful references. The author also gratefully acknowledges financial support from the EPSRC.
\section{Notation and prerequisites}
We will use the asymptotic notation $f=O(g)$ or $f\ll g$ if there is a constant $C$ such that $|f(n)|\leqslant C g(n)$ for all $n$, $f=\Theta(g)$ if there exist constants $C,C'$ such that $C'g(n)\leqslant f(n)\leqslant Cg(n)$ for all $n$, and we write $f=o(g)$ if $\frac{f(n)}{g(n)}\to 0$ as $n\to\infty$. Sometimes we emphasise the limit by writing $f=O_{n\to\infty}(g)$ or $f=o_{n\to\infty}(g)$. We also write $f=O_c(g)$ if the constant $C$ above is allowed to depend on $c$, and $f=o_c(g)$ if $\frac{f(n)}{g(n)}$ tends to $0$ as $n\to\infty$ for each fixed $c$ but at a rate that may depend on $c$. We say that $z\in\mathbf{C}$ is unimodular if $|z|=1$, and we write $\Re(z),\Im(z)$ for its real and imaginary parts. For a real number $t$, we use the notation $e(t)=e^{2\pi i t}$. We will make use of Parseval's theorem, which for the group $G=\mathbf{Z}/D\mathbf{Z}$ and functions $f,g:G\to \mathbf{C}$ states that
\begin{equation*}
    \sum_{m=0}^{D-1}\hat{f}(m)\overline{\hat{g}(m)}=D^{-1}\sum_{m=0}^{D-1}f(m)\overline{g(m)},
\end{equation*}
where $\hat{f}(m)= D^{-1}\sum_{j=0}^{D-1}f(j)e(-mj/D)$ denotes the Fourier transform of $f$. For the group $G=\mathbb{T}$ and a trigonometric polynomial $f(t)=\sum_{-N}^Nc_ne(nt)$, Parseval's theorem states that $\int_0^{1}|f(t)|^2\;dt=\sum_n|c_n|^2$.

\section{Set up and outline of the proof}
We prove Theorem \ref{maintheorem} by contradiction, assuming that there is a fixed integer $d$ and an infinite set $\mathcal{N}\subset\mathbf{N}$ such that for each $N\in\mathcal{N}$ we can find a reciprocal Littlewood polynomial $P_N(z)\in\mathcal{L}_N$ of degree $N$ that has at most $d$ unimodular zeros. Our starting point is a structural result on trigonometric polynomials with a bounded number of roots, appearing in the papers of Sahasrabudhe \cite{sahasrabudhe} and Erd\'elyi \cite{erdelyi}.
\begin{theorem}\cite[Theorem 4]{sahasrabudhe}. 
\label{boundedcomplexity} Let $R\subset \mathbf{Z}$ be finite and $d\in \mathbf{N}$. Then there is a constant $D$ depending only on $R,d$ such that the following holds. Let $$f(\theta)=\sum_{n=0}^NA_n\cos n\theta+\sum_{n=1}^NB_n\sin n\theta$$ be a trigonometric polynomial with coefficients $A_n,B_n\in R$ having at most $d$ roots in $[0,2\pi]$. Then there exists $L<D$ and disjoint intervals $I_j$ with $\cup_{j=1}^LI_j=[N]$ so that for each $j$, the sequences $(A_n)_{n\in I_j},(B_n)_{n\in I_j}$ are periodic with period $D$.
\end{theorem}
The unimodular roots of a reciprocal polynomial $P(z)=\sum_{n=0}^{N}a_nz^n$ with real coefficients correspond to the roots of a cosine polynomial via the relations 
\begin{align}
P(e^{i\theta})e^{-i\theta  N/2}&=a_{N/2}+2\sum_{n=1}^{N/2}a_{N/2+n}\cos n\theta \; &&\text{for $N$ even,} \label{Ptocosine}\\
P(e^{2i\theta})e^{-iN\theta}&=2\sum_{n=1}^{(N+1)/2}a_{(N-1)/2+n}\cos((2n-1)\theta) &&\text{for $N$ odd}.
\nonumber
\end{align}
The polynomials $P_{N}$ with $N\in\mathcal{N}$ are reciprocal Littlewood polynomials by assumption so we can write 
\begin{equation}
    P_{N}(z)=\sum_{n=0}^{N}a^{(N)}_nz^n
\label{P_Ndefi}
\end{equation} with coefficients $a^{(N)}_n\in \{\pm1\}$ that satisfy $a^{(N)}_{n}=a^{(N)}_{N-n}$. Let us define for each $N\in\mathcal{N}$ a corresponding polynomial $Q_{M}$ of degree $M$ by
\begin{align}
Q_{N/2}(z)&\vcentcolon=\sum_{n=0}^{N/2}a^{(N)}_{N/2+n}z^n \; &&\text{if $N$ is even,}\label{Qdefi}\\
Q_{N}(z)&\vcentcolon=\sum_{n=1}^{(N+1)/2}a^{(N)}_{(N-1)/2+n}z^{2n-1} &&\text{if $N$ is odd.}\label{Qdefiodd}
\end{align}
Note that $Q_{N/2}$ is itself a Littlewood polynomial when $N$ is even. In this case, we will show that there exist large $N\in\mathcal{N}$ for which the function $\Re\left(Q_{N/2}(e^{i\theta})\right)$ oscillates sufficiently many times between $-cN$ and $cN$ on $[0,2\pi]$ for some constant $c>0$. This is enough to establish the conclusion of Theorem \ref{maintheorem} for even degrees $N$ as from \eqref{Ptocosine} we obtain the following relation between $P_N$ and $Q_{N/2}$ for even $N$:
\begin{align}
    P_{N}(e^{i\theta}) &= e^{i\theta N/2}\Re\left(2Q_{N/2}(e^{i\theta})-a_{N/2}\right).
\label{QtoPeven}
\end{align} The situation for 
odd $N$ is more delicate; in this case 
\begin{align}
    P_{N}(e^{2i\theta}) &= 2e^{i\theta N}\Re\left(Q_{N}(e^{i\theta})\right)
\label{QtoPodd}
\end{align}
but we cannot show that a function $\Re\left(Q_{N}(e^{i\theta})\right)$ of the form \eqref{Qdefiodd} has many roots by proving that it has large oscillations. The functions $h_m(\theta)=\Re\left(\sum_{n=0}^{2m}(-1)^{n}e^{i(2n+1)\theta}\right)$ which are similar to \eqref{fewrootspolys} are an example of this and they even have a bounded number of sign changes on $[0,2\pi]$ for all $m$. We will show that for odd $N$, either $\Re\left(Q_{N}(e^{i\theta})\right)$ has large oscillations so that we may conclude as before, or else we deduce a stronger periodicity condition for its coefficients than one gets from Theorem \ref{boundedcomplexity}. This additional structure will allow us to conclude by a different argument.

\section{A structural result}
As a first step, we prove the following lemma which shows that a sequence of polynomials with a bounded number of roots contains a subsequence of polynomials that all have the same structure.
\begin{lemma} \label{structurelem}
    Let $d\in\mathbf{N}$ and let $R\subset \mathbf{Z}$ be finite. Then there exists a constant $D$ depending only on $R,d$ such that the following holds. Let $\mathcal{N}\subset \mathbf{N}$ be an infinite set. Suppose that $\{Q_N(z):N\in\mathcal{N}\}$ is a collection of polynomials with coefficients in $R$, with $\deg Q_N=N$ and such that $\Re\left( Q_N(e^{i\theta})\right)$ has at most $d$ roots in $[0,2\pi]$ for each $N\in\mathcal{N}$. Then there exist an infinite subset $\mathcal{N}'\subset \mathcal{N}$, constants $l\leqslant D$ and $0=\rho_0<\rho_1<\dots<\rho_l=1$, and a sequence $(\varepsilon^{(j)}_m)_{m=1}^D\in R^D$ for each $j\in [l]$ so that for $N\in\mathcal{N}'$:
\begin{equation}
        Q_{N}(z) = \sum_{j=1}^{l}\left(\sum_{n\in \big(\rho_{j-1}N,\rho_jN\big)}\varepsilon^{(j)}_{n \,(\!\!\!\!\!\!\mod D)}z^n\right)+ E_N(z),
\label{lemmastructure}
    \end{equation}
where the $E_N(z)$ are polynomials with $o(N)$ terms and whose coefficients have size at most $2\max_{r\in R}|r|$.
\end{lemma}
\begin{rmk*}
    We remind the reader that here, $o(N)$ denotes a function $f$ defined on $\mathcal{N}'$ such that $f(N)/N\to 0$ as $N\to\infty,N\in\mathcal{N}'$.
\end{rmk*}
\begin{proof}
Let us write $Q_{N}(z)=\sum_{n=0}^{N}q^{(N)}_nz^n$ for each $N\in\mathcal{N}$ so that $q^{(N)}_n\in R$ and $\deg Q_N=N$. Our aim is to prove \eqref{lemmastructure} for all $N$ in some infinite subset of $\mathcal{N}$, so we may assume throughout this proof that $N$ is sufficiently large. As we are assuming that the cosine polynomial $\Re(Q_{N})$ has at most $d$ roots, an application of Theorem \ref{boundedcomplexity} shows that for each $N\in\mathcal{N}$ the sequence of coefficients $(q^{(N)}_n)_{0\leqslant n\leqslant N}$ consists of a bounded number $L_N<D=O_{R,d}(1)$ of periodic sequences, each periodic with period $D$. Let us write $I^{(N)}_j$ for the intervals partitioning $[N]$ such that $(q^{(N)}_n)_{n\in I^{(N)}_j}$ is $D-$periodic for each $1\leqslant j\leqslant L_N$. By adding at most $D$ intervals, we may assume that the startpoints of the $I^{(N)}_j$ are multiples of $D$ and hence, so are their lengths. Since there are only finitely many values of $L_N$, we may restrict to an infinite subset $\mathcal{N}_1\subset \mathcal{N}$ so that $L_N=L$ is constant for each $N\in\mathcal{N}_1$. Note further that there are only finitely many ways to choose $L$ sequences in $R^D$. Hence, after further restricting to an appropriate subset $\mathcal{N}_2\subset\mathcal{N}_1$, we may also assume that the $D-$periodic pattern of the sequences $(q^{(N)}_n)_{n\in I^{(N)}_j}$ depends only on $j\in [L]$ but not on $N$ when $N\in\mathcal{N}_2$. Formally, this means precisely that we can find fixed sequences $(\epsilon^{(j)}_m)_{m=1}^{D}$ in $R^D$, one for each $j\in [L]$, so that for all $N\in\mathcal{N}_2$ we have 
\begin{equation}
q^{(N)}_n=\epsilon^{(j)}_{n (\!\!\!\!\!\!\mod D)}
\label{epsilondefilemma}
\end{equation} when $n\in I^{(N)}_j$. Therefore, when $N\in\mathcal{N}_2$, we can write
\begin{equation}
    Q_N(z) = \sum_{j=1}^L\left(\sum_{n\in I^{(N)}_j}\epsilon^{(j)}_{n \,(\!\!\!\!\!\!\mod D)}z^n\right).
\label{Q_Nlemmafull}
\end{equation}
Let us write the intervals as $I^{(N)}_j=\left[s^{(N)}_{j-1},s^{(N)}_j\right)$ so $0=s^{(N)}_0<\dots <s^{(N)}_L=N+1$. By a compactness argument and by again restricting to an infinite subset $\mathcal{N}_3\subset \mathcal{N}_2$, we may assume that the sequence of vectors $\sigma_N=\left(\frac{s^{(N)}_0}{N},\frac{s^{(N)}_1}{N},\dots,\frac{s^{(N)}_L}{N}\right)$ converges to $\sigma=(\sigma_0,\sigma_1,\dots,\sigma_L)\in[0,1]^{L+1}$ as $N\to \infty,N\in\mathcal{N}_3$. It is clear that $0=\sigma_0\leqslant \sigma_1\leqslant \dots\leqslant \sigma_L=1$. Let us now define $\rho=(\rho_0,\rho_1,\dots,\rho_l)\in[0,1]^{l+1}$ to be the vector obtained from $\sigma$ by taking $\rho_0=\sigma_0=0$ and for $i\geqslant 1$ setting $\rho_i = \min\{\sigma_j: \sigma_j>\rho_{i-1}\}$, so one gets $\rho$ by deleting repeated entries in the vector $\sigma$. Let us also define for each $j\in[l]$ the index $k_j$ to be the least integer $k_j\in[L]$ for which $\sigma_{k_j}=\rho_j$. Hence, if $N\in\mathcal{N}_3$ it follows that for each $j\in[l]$:
\begin{equation}
    I^{(N)}_{k_j}=\left(\rho_{j-1}N+o(N),\rho_jN+o(N)\right)
\label{denseintervals}
\end{equation}
and that $\left|I^{(N)}_i\right|=o(N)$ if $i\in[L]\setminus\{k_1,\dots,k_l\}$.
When $N\in\mathcal{N}_3$, we may write
\begin{equation}
    Q_{N}(z) = \sum_{j=1}^{l}\left(\sum_{n\in \big(\rho_{j-1}N,\rho_jN\big)}\epsilon^{(k_j)}_{n \,(\!\!\!\!\!\!\mod D)}z^n\right)+ E_N(z),
\end{equation}
where we observe from \eqref{Q_Nlemmafull} that if
\begin{align*}
E_N(z)&= Q_N(z)-\sum_{j=1}^{l}\left(\sum_{n\in \big(\rho_{j-1}N,\rho_jN\big)}\epsilon^{(k_j)}_{n \,(\!\!\!\!\!\!\mod D)}z^n\right)
\label{E_Nform}
\end{align*} has a term $z^n$ with a non-zero coefficient, then $n\notin \cup_{j=1}^l\left(I^{(N)}_{k_j}\cap (\rho_{j-1}N,\rho_j N)\right)$. By \eqref{denseintervals} there are at most $o(N)$ such integers $n\in[N]$ so it follows that the $E_N(z)$ are polynomials with $o(N)$ terms and coefficients of size at most $2\max_{r\in R} |r|$. Hence, if we define the sequences $(\varepsilon^{(j)}_m)=(\epsilon^{(k_j)}_m)$ for $j$ in $[l]$, then we obtain \eqref{lemmastructure}.
\end{proof}
For the final section of this paper, it is convenient to have following result which describes the structure of the polynomials $Q_N$ in a more detailed manner than Lemma \ref{structurelem} does, under the assumption that all the sequences $(\varepsilon^{(j)}_m)_{m=1}^D$ in \eqref{lemmastructure} are identical.
\begin{cor} \label{coroflem1}
Let $D=D(R,d)$ be the constant from Lemma \ref{structurelem}, let $\mathcal{N}'\subset \mathbf{N}$ be an infinite set and let $\{Q_N:N\in\mathcal{N}'\}$ be a collection of polynomials satisfying the conclusion of Lemma \ref{structurelem}. Suppose that the sequences $(\varepsilon^{(j)}_m)_{m=1}^D\in R^D$ in \eqref{lemmastructure} are identical for $j\in [l]$ and let $(\varepsilon_m)_{m=1}^{D}\in R^D$ be this common sequence. Then there exist constants $L'\leqslant D$ and $0\leqslant \tau_0\leqslant\dots\leqslant\tau_{L'}\leqslant 1$, sequences $(\delta^{(j)}_m)_{m=1}^D$ for $j\in[L']$, and for each $N\in\mathcal{N}'$ a collection of $L'$ disjoint intervals $J^{(N)}_j\subset [N]$, $j\in [L']$, such that the following holds. 
\begin{enumerate}
    \item[(i)] The starting point and the length of each of the intervals $J^{(N)}_j$ is a multiple of $D$, and $\min J^{(N)}_j,\max J^{(N)}_j=\tau_j N+o(N)$ for each $j\in[L']$. In particular, the intervals have length $\left|J^{(N)}_j\right|=o_{N\to\infty,N\in\mathcal{N}'}(N)$.
    \item[(ii)] For each $N\in\mathcal{N}'$, we can write
\begin{equation*}
    Q_{N}(z)= \left(\sum_{m=0}^{D-1} \varepsilon_mz^m\right)\frac{z^{N+1}-1}{z^D-1}+\sum_{j=1}^{L'}\left(\sum_{m\in J^{(N)}_j}\delta^{(j)}_{m\,(\!\!\!\!\!\!\mod D)}z^m\right).
\end{equation*}
\end{enumerate}
\end{cor}
\begin{proof}
This is essentially immediate from the proof of Lemma \ref{structurelem} by taking for each $N\in\mathcal{N}'$ the intervals $J^{(N)}_j$ to be all of the intervals $I^{(N)}_i$ for which $i\in[L]\setminus\{k_1,\dots,k_l\}$, using the notation from the proof of Lemma \ref{structurelem}. Indeed, to see that item (i) holds, note that we chose the intervals $I^{(N)}_i$ in Lemma \ref{structurelem} to have starting points that are multiples of $D$ so that their lengths are multiples of $D$ and hence the same is true for the intervals $J^{(N)}_j$. Next, recall that $I^{(N)}_{k_j}=\left(\rho_{j-1}N+o(N),\rho_j N+o(N)\right)$ for $j\in [l]$ by \eqref{denseintervals}. Hence, each of the intervals $J^{(N)}_j$, which is of the form $I^{(N)}_i$ with $i\in[L]\setminus\{k_1,\dots,k_l\}$, is contained in an interval of the form $\left(\rho_{j}N+o(N),\rho_jN+o(N)\right)$ so that $\min J^{(N)}_j,\max J^{(N)}_j=\tau_jN+o(N)$ for some $\tau_j\in\{\rho_0,\dots,\rho_l\}$. Finally, note that item (ii) follows directly from \eqref{Q_Nlemmafull} and the assumption of Corollary \ref{coroflem1} that all the sequences $(\epsilon^{(k_j)}_m)$ are equal to $(\varepsilon_m)$ for $j\in[l]$, if we define, corresponding to each interval $J^{(N)}_j=I^{(N)}_i$ where $i\in[L]\setminus\{k_1,\dots,k_l\}$, the sequence $(\delta^{(j)}_m)_{m=1}^D$ by 
\begin{equation}
    \delta^{(j)}_m= \epsilon^{(i)}_m-\varepsilon_m.
\label{deltafirstdefi}
\end{equation}
\end{proof}
The following lemma will be useful at several stages in our argument. It allows us to show that a sequence of trigonometric functions with a certain common structure contains a subsequence of functions that all behave, on an appropriate scale, like the same fixed function.
\begin{lemma} \label{lem2}
    Let $\mathcal{N}\subset \mathbf{N}$ be an infinite set, let $R\subset\mathbf{C}$ be finite, let $0\leqslant \rho_0<\dots<\rho_l\leqslant 1$ and let $l,D\in\mathbf{N}$. Let $\{H_N(\theta):[0,1]\to\mathbf{C}\}_{N\in\mathcal{N}}$ be a set of functions. Suppose that for each $N\in\mathcal{N}$ there are $l+1$ trigonometric polynomials $B^{(N)}_j(\theta)$ with coefficients in $R$, with at most $D$ non-zero coefficients, and of degree $\deg B^{(N)}_j = o(N)$ so that for each $N\in\mathcal{N}$ we can write
    \begin{equation}
        H_N(\theta) = \sum_{j=0}^lB^{(N)}_j(\theta)e(r^{(N)}_j\theta),
    \end{equation}
    where the $r^{(N)}_j$ are real numbers satisfying $r^{(N)}_j=\rho_jN+o(N)$. If $(\gamma_N)_{N\in\mathcal{N}}$ is a sequence of points $\gamma_N\in[0,1]$, then there exists an infinite subset $\mathcal{N}'\subset\mathcal{N}$ and constants $b_j\in\mathbf{C}$ such that the following holds. Let $C>0$, then for all $N\in\mathcal{N}'$ and $u\in(-C,C)$ we have
    \begin{equation}
    H_N\left(\gamma_N+\frac{u}{N}\right)=\sum_{j=0}^lb_je(\rho_j u) +o_{C;N\to\infty}(1).
\label{lemma2Q}
    \end{equation}
    Moreover, for each $j$ we have
    \begin{equation}
        \lim_{N\to\infty,N\in\mathcal{N}'} B^{(N)}_j(\gamma_N)e(r^{(N)}_j\gamma_N)=b_j.
    \label{b_jlem2size}
    \end{equation}
\end{lemma}
\begin{rmk*}
    Here, we use the notation $o_{C;N\to\infty}(1)$ for a function $f(N,u)$ that satisfies $\sup_{u\in(-C,C)}|f(N,u)|\to 0$ as $N\to\infty,N\in \mathcal{N}'$.
\end{rmk*}
\begin{proof}
    This follows from the following elementary compactness argument. Note that each trigonometric polynomial $B^{(N)}_j$ has at most $D$ non-zero terms and coefficients in a finite set $R$ so that the vectors $(B^{(N)}_j(\gamma_N)e(r^{(N)}_j\gamma_N))_{j=0}^l$ with $N\in\mathcal{N}$ form a bounded sequence in $\mathbf{C}^{l+1}$. By compactness, there exists an infinite set $\mathcal{N}'\subset\mathcal{N}$ and constants $b_j\in\mathbf{C}$ such that $B^{(N)}_j(\gamma_N)e(r^{(N)}_j\gamma_N)\to b_j$ as $N\to\infty,N\in\mathcal{N}'$. Let $C>0$ be a fixed constant. Since each $B^{(N)}_j$ has degree $o(N)$ and at most $D$ terms with coefficients in $R$, we can use the basic bound $|e(\alpha)-e(\beta)|\ll |\alpha-\beta|$ for $\alpha,\beta\in\mathbf{R}$ to deduce that for all $u\in(-C,C)$:\begin{align*}
        &\left|B^{(N)}_j(\gamma_N+u/N)e(r^{(N)}_j\gamma_N) - B^{(N)}_j(\gamma_N)e(r^{(N)}_j\gamma_N)\right|\\
        &= \left|B^{(N)}_j(\gamma_N+u/N)- B^{(N)}_j(\gamma_N)\right|\ll_{R,D,C}\deg B^{(N)}_j/N=o_{C;N\to\infty}(1)
    \end{align*} and hence that $H_N(\gamma_N+u/N)=\sum_jb_je(u r^{(N)}_j/N)+o_C(1)$ uniformly for all $u\in(-C,C)$. Using now that $|e(ur^{(N)}_j/N)-e(u\rho_j)|\ll_C o(1)$ when $|u|<C$ by the assumption that $r^{(N)}_j=\rho_jN+o(N)$, it follows immediately that $H_N(\gamma_N+u/N)=\sum_{j=0}^lb_je(\rho_j u) +o_{C;N\to\infty}(1)$.
\end{proof}
We state one more lemma which provides us with a useful criterion for when a trigonometric function with real frequencies, like those appearing in the right hand side of \eqref{lemma2Q}, has many zeros.
\begin{lemma} \label{signchange}
    Let $H$ be a function of the form \begin{equation}
        H(u) = a_0+\sum_{j=1}^{l}\left(a_je(\rho_ju)+\overline{a_j}e(-\rho_ju)\right)
    \label{Hinlemma}
    \end{equation} with coefficients $a_0\in\mathbf{R}$, $a_j\in\mathbf{C}$ and $0<\rho_1<\dots<\rho_{l-1}<\rho_l=1$. Suppose that either $2|a_l|>|a_0|$, or else that $2|a_l|\geqslant|a_0|$ and $a_j\neq 0$ for some $j$, $1\leqslant j\leqslant l-1$. Then $H$ has infinitely many sign changes on $\mathbf{R}$.
\end{lemma}  To prove this lemma, we use the following classical theorem of Fej\'er \cite{fejer} and Riesz \cite{riesz}, see also \cite[page 26]{BarryforFEJERRIESZ} for a more modern discussion.
\begin{theorem}[Fej\'er-Riesz Theorem]
Let $w(t)=\sum_{n=-N}^Nc_ne^{itn}$ be a real-valued trigonometric polynomial that is non-negative for all real $t$. Then $w$ is expressible in the form
$w(t)=|p(e^{it})|^2$ for some polynomial $p(z)=\sum_{n=0}^Nd_nz^n$ with $d_n\in\mathbf{C}$.   
\end{theorem}
\begin{cor} \label{cor1}
Let $w(t)=\sum_{n=-N}^Nc_ne^{itn}$ be a real-valued trigonometric polynomial. Suppose that either $2|c_N|>|c_0|$, or else that $2|c_N|=|c_0|$ and that $c_n\neq 0$ for some $1\leqslant n\leqslant N-1$. Then $w$ has a sign change in $[0,2\pi]$.
\end{cor}
\begin{proof}
We argue by contradiction, assuming that there is a choice of sign $\epsilon\in\{\pm1\}$ such that function $ \epsilon w(t)$ is non-negative. By Fej\'er-Riesz, there then exist $d_j\in\mathbf{C}$ such that \begin{equation}
    \epsilon c_n = \sum_{j=0}^{N-n}d_j\overline{d_{n+j}}.
    \label{fejerrieszrep}
\end{equation}
Hence, $|c_N|=|d_0d_N|\leqslant \frac{1}{2}\sum_{j=0}^N|d_j|^2=\frac{|c_0|}{2}$ so the inequality $|c_N|>\frac{|c_0|}{2}$ is false and the assumptions of the lemma then imply that $c_n\neq 0$ for some $n\in[N-1]$. Since $|c_N|\geqslant\frac{|c_0|}{2}$ by assumption, this chain of inequalities also shows that $\sum_{j=1}^{N-1}|d_j|^2=0$. Hence, we deduce from \eqref{fejerrieszrep} that $c_n=0$ for all $1\leqslant n\leqslant N-1$, a contradiction.
\end{proof}
\begin{proof}[Proof of Lemma \ref{signchange}]
To apply Fej\'er-Riesz, we write \eqref{Hinlemma} as  $H(u)= w_1(u)+w_2(u)$, where 
\begin{align}
    w_1(u)= a_0+\sum_{j:\rho_j\in\mathbf{Q}}\left(a_je(\rho_ju)+\overline{a_j}e(-\rho_ju)\right)
\label{Hdecompo}
\end{align}
is a rescaled real-valued trigonometric polynomial, and \begin{align*}
w_2(u)=\sum_{j:\rho_j\notin\mathbf{Q}}\left(a_je(\rho_ju)+\overline{a_j}e(-\rho_ju)\right). 
\end{align*}
By the assumptions of Lemma \ref{signchange}, we either have that $2|a_l|> |a_0|$, or else that $2|a_l|=|a_0|$ and $a_j\neq 0$ for some $j \in[l-1]$. We consider two cases. In the first case we assume that $2|a_l|> |a_0|$, or else that $2|a_l|=|a_0|$ and $a_j \neq 0$ for some $j$ with $\rho_j\in \mathbf{Q}\setminus\{0,1\}$. In the second case, we have that $2|a_l|=|a_0|$, that $a_j = 0$ for all $j$ with $\rho_j\in \mathbf{Q}\setminus\{0,1\}$ and that 
$w_2$ has a non-zero coefficient. The assumptions of the first case allow us to apply Corollary \ref{cor1} and deduce that $w_1$ has a sign change, so there exist $\theta_{+},\theta_{-}$ such that $w_1(\theta_\pm)=\pm\kappa$, for some $\kappa>0$. Note from the definition \eqref{Hdecompo} of $w_1$ that $w_1$ is $p$-periodic for some $p\in\mathbf{N}$. As $\sum_{m=1}^n e(p m\rho)= O_{n\to\infty}(1)$ for each fixed irrational $\rho$ and as all exponentials in $w_2$ have irrational frequencies, we get that 
\begin{equation*}
    \sum_{m=1}^nH(\theta_\pm +pm)=\sum_{m=1}^n\left(w_1(\theta_\pm+pm)+w_2(\theta_\pm+ pm)\right) = \pm \kappa n+O_{n\to\infty}(1)
\end{equation*} showing that $H(u)>\kappa/2$ and $H(u')<-\kappa/2$ for arbitrarily large $u,u'$. This implies the desired conclusion of Lemma \ref{signchange} that $H$ oscillates between $-\kappa/2$ and $\kappa/2$ infinitely often on $\mathbf{R}$. Now assume that we are in the second case. So $2|a_l|=|a_0|$, $a_j = 0$ for all $j$ with $\rho_j\in \mathbf{Q}\setminus\{0,1\}$, and $w_2$ has a non-zero coefficient. Note that in this case, \eqref{Hdecompo} becomes $w_1(u)= a_0+a_le(u)+\overline{a_l}e(-u)$ and as $2|a_l|=|a_0|$ it is clear that $w_1$ has a double root at $\theta = 1/2-\arg(a_l/a_0)/2\pi$. By a similar calculation as above, we therefore have
\begin{equation}
    \sum_{m=1}^nH(\theta+m) = O_{n\to\infty}(1)
    \label{boundedsums}
\end{equation} as $w_1$ is $1-$periodic. Consider now \begin{align*}
    w_2^2(u)=&\sum_{j,j':\rho_j,\rho_{j'}\notin\mathbf{Q}}\left(a_ja_{j'}e((\rho_j+\rho_{j'})u)+\overline{a_j}\overline{a_{j'}}e((-\rho_j-\rho_{j'})u)\right)\\
    +&\sum_{j,j':\rho_j,\rho_{j'}\notin\mathbf{Q}}\left(a_{j}\overline{a_{j'}}e((\rho_j-\rho_{j'})u)+\overline{a_j}a_{j'}e((\rho_{j'}-\rho_j)u)\right)
\end{align*} and as $0<\rho_j<1$ for all irrational $\rho_j$, the only terms in $w_2^2$ having frequencies that are integer multiples of $2\pi$ come from the diagonal terms $j=j'$ and some pairwise disjoint pairs $\{j,j'\}$ for which $\rho_j+\rho_{j'}=1$. So we can partition $\{j:\rho_j\notin \mathbf{Q}\}$ into a set $\Lambda_1$ of pairs $\{j,j'\}$ for which $\rho_j+\rho_{j'}=1$ and a set $\Lambda_2$ of the remaining $j$. As $w_1(\theta+m)=0$, applying the basic bound that $\sum_{m=1}^ne(m\alpha)= O_{n\to\infty}(1)$ for each fixed $\alpha\in\mathbf{R}\setminus{\mathbf{Z}}$ to the terms in $w_2^2$ shows that
\begin{align}
    \sum_{m=1}^nH(\theta+ m)^2 &= n\sum_{j:\rho_j\notin \mathbf{Q}}|a_j|^2+n\sum_{\{j,j'\}:\rho_j+\rho_{j'}=1}\left(a_ja_{j'}e(\theta)+\overline{a_j}\overline{a_{j'}}e(-\theta)\right)+O_{n\to\infty}(1)\nonumber\\
    &=n\sum_{j\in\Lambda_2}|a_j|^2+n\sum_{\{j,j'\}\in\Lambda_1}\left|a_je(\theta/2)+\overline{a_{j'}}e(-\theta/2)\right|^2+O_{n\to\infty}(1).
    \label{secondmomentG}
\end{align}
Hence, if $\Lambda_2$ is non-empty, or if $a_je(\theta/2)+\overline{a_{j'}}e(-\theta/2)\neq0$ for some $\{j,j'\}\in\Lambda_1$ then \eqref{secondmomentG} is $\gg n$ so that we can use the following second moment argument to show that $H(\theta+m)$ takes both positive and negative values for arbitrarily large $m$, implying the desired conclusion of Lemma \ref{signchange}. Indeed, if for a contradiction we assume that $H(\theta+m)$ has a constant sign for $m>n_0$, then 
\begin{align*}
  \sum_{m=1}^nH(\theta+m)^2 = O_{n\to\infty}(1)+O_{n\to\infty}\left(||H||_\infty \left|\sum_{n_0<m\leqslant n}  H(\theta+m)\right|\right) = O_{n\to\infty}(1)
\end{align*}
since $H$ is a finite sum of exponentials so bounded on $\mathbf{R}$ and by \eqref{boundedsums}, but this contradicts \eqref{secondmomentG} as $n\to\infty$. Finally, suppose that $\Lambda_2=\emptyset$ and that $a_je(\theta/2)+\overline{a_{j'}}e(-\theta/2)=0$ for all $\{j,j'\}\in\Lambda_1$. Then \eqref{secondmomentG} is $O_{n\to\infty}(1)$. We repeat the calculation above for $H'(u)$ to obtain
\begin{align}
    \sum_{m=1}^nH'(\theta+m)^2 &= n\sum_{\{j,j'\}\in\Lambda_1}\left|2\pi i\rho_ja_je(\theta/2)-2\pi i\rho_{j'}\overline{a_{j'}}e(-\theta/2)\right|^2+O_{n\to\infty}(1)\gg n,
    \label{derivativebound}
\end{align}
where we used that $w_1'(\theta+m)=0$ as $w_1$ has a double root at $\theta$, and for the final inequality that $\rho_ja_je(\theta/2)-\rho_{j'}\overline{a_{j'}}e(-\theta/2)\neq 0$ if $a_je(\theta/2)+\overline{a_{j'}}e(-\theta/2)=0$. From \eqref{derivativebound}, we see that $H'(\theta+ m)\gg 1$ for arbitrarily large integers $m$, while the fact that \eqref{secondmomentG} is $O_{n\to\infty}(1)$ implies that $H(\theta+m)=o(1)$ as $m\to\infty$. Note that $H''(u)$ is a sum of a finite number of exponentials so bounded on $\mathbf{R}$. It is an elementary exercise in analysis that these three combined bounds on $H,H'$ and $H''$ imply that $H$ has infinitely many sign changes on $\mathbf{R}$; note that if $H'(\theta+m)> c$ and $||H''||_\infty<C$, then Taylor's theorem implies that $H(\theta+m+c/C)>H(\theta+m)+c^2/2C$ and $H(\theta+m-c/C)<H(\theta+m)-c^2/2C$. As we have shown that that there are arbitrarily large $m$ satisfying $H'(\theta+m)> c$, combined with the fact that $H(\theta+m)=o_{m\to\infty}(1)$ this shows that $H(\theta+m+c/C)>0$ while $H(\theta+m-c/C)<0$ for each such large $m$. Hence, $H(u)$ takes positive and negative values for arbitrarily large $u$, which is the desired conclusion of Lemma \ref{signchange}.
\end{proof}
\section{Proof for even degrees}
In this section we establish the result of Theorem \ref{maintheorem} for even degrees, so our goal is to show that $Z_\mathcal{L}(2N) \to\infty$ as $N\to\infty$. To prove this, recall from Sect. 3 that we assume for a contradiction that there is an infinite set of positive integers $\mathcal{N}\subset \mathbf{N}$ and for each $N\in\mathcal{N}$ a reciprocal Littlewood polynomial $P_{2N}$ of the form \eqref{P_Ndefi} and degree $2N$ which has at most $d$ unimodular roots. We write $Q_{N}=\sum_{n=0}^{N}q^{(N)}_nz^n$ for the corresponding polynomials that we defined in \eqref{Qdefi}, and recall the relation \eqref{QtoPeven} between $P_{2N}$ and $Q_N$. As we are assuming that each $P_{2N}$ has at most $d$ unimodular roots, we deduce that $\Re(Q_N(e^{i\theta})-\frac{1}{2}a_{N})$ also has at most $d$ roots in $[0,2\pi]$. Applying Lemma \ref{structurelem}, we obtain an infinite subset $\mathcal{N}'\subset \mathcal{N}$, constants $D,l=O_d(1)$ and $0=\rho_0<\dots<\rho_l=1$, and $l$ sequences 
\begin{equation}    \left(\varepsilon^{(j)}_m\right)_{m=1}^D\in\{\pm 1\}^D,
\label{epsilonvalues}
\end{equation} one for each $j\in[l]$, so that for each $N\in\mathcal{N}'$:
\begin{equation}
    Q_{N}(z) = \sum_{j=1}^{l}\left(\sum_{n\in \big(\rho_{j-1}N,\rho_jN\big)}\varepsilon^{(j)}_{n \,(\!\!\!\!\!\!\mod D)}z^n\right)+ E_N(z),
\label{firstsimpl}
\end{equation}
where the $E_N(z)$ are polynomials whose coefficients are $O(1)$ and with $o(N)$ terms. Here and in the rest of the paper, unless explicitly stated otherwise, the asymptotic notation describes the behaviour as $N\to\infty$. We rewrite this expression in a form that is more convenient. Note first that we may assume that $N\in\mathcal{N}'$ is sufficiently large since we only aim to find one polynomial $P_{2N}$ having more than $d$ unimodular roots in order to get the required contradiction. Let $r^{(N)}_j$ be the largest multiple of $D$ in $\big(\rho_{j-1}N,\rho_jN\big)$ for $j\in[l]$, and let $r^{(N)}_0=0$. Clearly $r^{(N)}_j=\rho_jN+O(1)$, so we can sum the geometric series in \eqref{firstsimpl} to get that for each $N\in\mathcal{N}'$:
\begin{align}
    Q_{N}(z) &= \frac{1}{z^D-1}\sum_{j=1}^{l}(z^{r^{(N)}_j}-z^{r^{(N)}_{j-1}})\left(\sum_{m=0}^{D-1}\varepsilon^{(j)}_{m}z^m\right)+ E'_N(z),
\nonumber\\
    &=\frac{1}{z^D-1}\sum_{j=1}^{l}A_j(z)(z^{r^{(N)}_j}-z^{r^{(N)}_{j-1}})+ E'_N(z), \label{furthersimpl+}
\end{align}
where again the $E'_N(z)$ are polynomials whose coefficients are $O(1)$ and with $o(N)$ terms, and where for $1\leqslant j \leqslant l$ we defined
\begin{equation}
    A_j(z) = \sum_{m=0}^{D-1}\varepsilon^{(j)}_mz^j.
    \label{A_jdefinitions}
\end{equation} An important consequence of Lemma \ref{structurelem} is that the $A_j$ are fixed polynomials, independent of $N\in\mathcal{N}'$. We need the following somewhat technical result to deduce that the functions $Q_{N}$ have large oscillations.
\begin{proposition} \label{prop1}
Let $l,D\in\mathbf{N}$, let $0=\rho_0<\dots<\rho_l=1$ and let $R\subset \mathbf{C}$ be finite. Suppose that $B_j(\theta)$ is a trigonometric polynomial of degree at most $D$ and with coefficients in $R$ for $0\leqslant j\leqslant l$. Define for each large $N\in\mathbf{N}$ the integer $r^{(N)}_j$ to be the largest multiple of $D$ in $(\rho_{j-1}N,\rho_jN)$, and let $r^{(N)}_0=0$ for all $N$. Define for each $N\in\mathbf{N}$ the function $f_N$ by 
\begin{align}
     f_N(\theta)\vcentcolon=\Re\left(\frac{1}{e(\theta D)-1}\sum_{j=1}^{l}B_j(\theta)\big(e(\theta r^{(N)}_j)-e(\theta r^{(N)}_{j-1})\big)\right).
\label{f_Ndefi}
\end{align} Suppose that there exists an integer $m$ such that
\begin{align}
    \left|B_l(m/D)\right|\geqslant\left|\Im\left(B_1(m/D)\right)\right|,
\label{prop1cond}
\end{align}
and if equality holds, then in addition 
\begin{align}
    B_{j+1}(m/D)\neq B_j(m/D)
\label{prop1cond2}
\end{align} for some $j$, $1\leqslant j\leqslant l-1$. Let $C>0$ be given. Then there exists a $c=c(C)>0$ such that for all sufficiently large $N$, the function $f_N$ oscillates between $-cN$ and $cN$ at least $C$ times on $[0,1]$.
\end{proposition}
Here, we say that a function $f$ oscillates between $C_1$ and $C_2$ at least $C$ times if there exist $\theta_1<\theta_2<\dots<\theta_{2C}$ such that $f(\theta_{2j-1})=C_1$ and $f(\theta_{2j})=C_2$ for all $j$. Before proving this proposition, we show how it may be used to deduce that $Z_{\mathcal{L}}(2N)\to \infty$ as $N\to \infty$.
\begin{proof}
Recall that we are supposing for a contradiction that there is an infinite set $\mathcal{N}\subset \mathbf{N}$ and for each $N\in\mathcal{N}$ a reciprocal Littlewood polynomial $P_{2N}$ of the form \eqref{P_Ndefi} with degree $2N$ having at most $d$ unimodular roots. By \eqref{Ptocosine} we have
\begin{align*}
    P_{2N}(e^{i\theta}) &= e^{i\theta N}\Re\left(2Q_{N}(e^{i\theta})-a_{N}\right),
\end{align*}
so to get a contradiction it suffices to prove that $\Re\left(2Q_{N}(e^{i\theta})-a_{N}\right)$ has more than $d$ roots for some large $N\in\mathcal{N}'$.
By \eqref{furthersimpl+}, there is an infinite subset $\mathcal{N}'\subset\mathcal{N}$ so that for each $N\in\mathcal{N}'$ we can write
\begin{align}
    \Re\left(2Q_{N}(e(\theta))-a_{N}\right) &= 2\Re\left(\frac{1}{e(\theta D)-1}\sum_{j=1}^{l}A_j(e(\theta))(e(\theta r^{(N)}_j)-e(\theta r^{(N)}_{j-1}))\right) +E'_N(\theta),
\label{Qevenproof}
\end{align}
where $E'_N(\theta)$ is a trigonometric polynomial with maximum value $||E'_N||_\infty =o(N)$. We show that condition \eqref{prop1cond} holds with strict inequality for the polynomials $A_j$. Let $\zeta=e(1/D)$ be a primitive $D$th root of unity. By Parseval's identity and the definition \eqref{A_jdefinitions} of $A_j$, we have
\begin{align*}
    \sum_{m=0}^{D-1}\left|A_l(\zeta^{m})\right|^2&=D\sum_{m=0}^{D-1}\left|\varepsilon^{(l)}_m\right|^2 = D^2
\end{align*}
since $\varepsilon^{(l)}_{m}=\pm1$ by \eqref{epsilonvalues}. Similarly
\begin{align*}
    \sum_{m=0}^{D-1}\left|\Im\left(A_1(\zeta^{m})\right)\right|^2 
    &= D\sum_{m=1}^{D-1}\left(\frac{\varepsilon^{(1)}_{m}-\varepsilon^{(1)}_{D-m}}{2}\right)^2\leqslant D(D-1).
\end{align*}
Hence, there exists some $m$ for which $\left|A_l(\zeta^{m})\right|>\left|\Im\left(A_1(\zeta^{m})\right)\right|$ so that the condition of Proposition \ref{prop1} is satisfied. Proposition \ref{prop1} gives a $c=c(2d)>0$ such that for $N$ sufficiently large, the function
\begin{align*}
    \theta \mapsto \Re\left(\frac{1}{e(\theta D)-1}\sum_{j=1}^{l}A_j(e(\theta))(e(\theta r^{(N)}_j)-e(\theta r^{(N)}_{j-1}))\right)
\end{align*}
oscillates between $-cN$ and $cN$ at least $2d$ times. Since $||E'_N(\theta)||_\infty=o(N)$ as $N\to\infty,N\in\mathcal{N}'$, from \eqref{Qevenproof} we get the required contradiction that $\Re\left(2Q_{N}(e^{i\theta})-a_{N}\right)$ has at least $2d$ roots for  sufficiently large $N\in\mathcal{N}'$.
\end{proof}
The same argument also establishes Theorem \ref{cosinetheorem} if one uses the relation \eqref{Ptocosine} to reduce the problem for cosine polynomials to an equivalent problem for reciprocal polynomials of \emph{even} degree. The polynomials that one gets are reciprocal Littlewood polynomials of even degree up to a constant term that is $O(1)$, so this makes no difference in the proof above. We now move on to proving Proposition \ref{prop1}.
\begin{proof}[Proof of Proposition \ref{prop1}]
Let $m$ satisfy \eqref{prop1cond}, and if there is equality in \eqref{prop1cond} then let \eqref{prop1cond2} hold in addition. Let $C>0$ be given, and let $f_N$ be the function defined in \eqref{f_Ndefi} so that our aim is to show that the $f_N$ have at least $C$ oscillations of size $\gg N$ for large $N$. We will consider angles $\theta \in\left(\frac{m}{D}-\frac{C'}{N},\frac{m}{D}+\frac{C'}{N}\right)=\vcentcolon I_N$, where $C'>0$ is a sufficiently large constant (depending on $C$) to be chosen later.
First observe that we can rewrite
\begin{align*}
    f_N(\theta)= \frac{1}{2\sin(2\pi D\theta)}\Im\left(\sum_{j=1}^{l}B_j(\theta)\left(e(\theta (r^{(N)}_j-D/2))-e(\theta (r^{(N)}_{j-1}-D/2)\right)\right).
\end{align*}
For all $\theta\in I_n$ we have that $|2\sin(2\pi D\theta)|\ll C'/N$.\footnote{The dependence on $m,l,D$ and $R$ will be suppressed throughout this argument as they are fixed as far as Proposition \ref{prop1} is concerned.} Hence, if we define the functions
\begin{align*}
    H_N(\theta)=\Im\sum_{j=1}^{l}B_j(\theta)\left(e(\theta (r^{(N)}_j-D/2))-e(\theta(r^{(N)}_{j-1}-D/2))\right)
\end{align*}
then it suffices to show that there exists a constant $c>0$ such that the $H_N$ oscillate at least $C$ times between $-c$ and $c$ on the interval $I_N$, as this would imply that the $f_N$ have $C$ oscillations of size $\gg_{c,C'}N$ on $I_N$. Combining the terms in $H_N$ with equal frequencies, we can rewrite
\begin{align*}
    H_N(\theta)&=\Im\left(B_1(\theta)+B_l(\theta)e(\theta( r^{(N)}_l-D/2))\right)\\
    &+\Im\sum_{j=1}^{l-1}\left(B_{j}(\theta)-B_{j+1}(\theta)\right)e(\theta (r^{(N)}_j-D/2)).
\end{align*}
Recall that by the assumptions of Proposition \ref{prop1}, the $B_j$ are trigonometric polynomials of degree at most $D=O(1)$ with coefficients in a finite set $R$, and also
that $r^{(N)}_j-D/2=\rho_jN+o(N)$. Hence, we may apply Lemma \ref{lem2} with the constant sequence $(\gamma_N)_{N\in\mathbf{N}}=(m/D)$ to obtain an infinite set $\mathcal{N}'\subset\mathbf{N}$ and constants $b_j\in\mathbf{C}$ such that for each $N\in\mathcal{N}'$ and $u\in(-C',C')$ we have\footnote{It is not hard to see from the proof of Lemma \ref{lem2} that we may in fact take $\mathcal{N}'=\mathbf{N}$ since the trigonometric polynomials $B_j$ appearing in Proposition \ref{prop1} do not depend on $N$.}
 \begin{equation}
    H_N\left(\frac{m}{D}+\frac{u}{N}\right)=\Im\left(\sum_{j=0}^{l}b_je(\rho_j u)\right)+o_{C';N\to\infty}(1).
\label{H_Napprprop1}
 \end{equation}
Let us now define
\begin{align}
    H(u)&= \Im\left(\sum_{j=0}^{l}b_je(\rho_j u)\right)=\Im(b_1)+\frac{1}{2i}\sum_{j=1}^l(b_je(\rho_ju)-\overline{b_j}e(-\rho_ju))
    \label{Hpolyform}
\end{align}
to be the function of $u$ appearing in the right hand side above. Note that $H(u)$ is a fixed function independent of $N$, so we have reduced our task to showing that $H$ has at least $C$ sign changes on $\mathbf{R}$. Indeed, if we can show this then there exists some $C'$ and $c>0$ such that $H$ oscillates between $-c$ and $c$ at least $C$ times on the interval $(-C',C')$. Hence, from \eqref{H_Napprprop1} we would deduce the desired result that the $H_N(\theta)$ oscillate between $-c/2$ and $c/2$ at least $C$ times on $I_N$ for large $N$. To accomplish this final task, we simply check that the conditions of Lemma \ref{signchange} are satisfied by the function $H$. As the trigonometric polynomials $B_j$ do not depend on $N$ and as we defined $r^{(N)}_0=0$ in the statement of Proposition \ref{prop1}, we see from equation \eqref{b_jlem2size} in Lemma \ref{lem2} that the coefficients of $H$ satisfy
\begin{align*}
    b_0&=B_1(m/D) \\
    |b_j|&=|B_{j+1}(m/D)-B_j(m/D)| &&\text{for $1\leqslant j\leqslant l-1$,} \\
    |b_l|&=|B_l(m/D)|
\end{align*}
Hence, the assumptions \eqref{prop1cond}, $\eqref{prop1cond2}$ of Proposition \ref{prop1} imply precisely that $|b_l|\geqslant|\Im(b_0)|$ and that if equality holds, then $b_j\neq 0$ for some $j\in[l-1]$. By Lemma \ref{signchange} these conditions imply that $H(u)$ has infinitely many sign changes on $\mathbf{R}$, so this completes the proof of Proposition \ref{prop1}.
\end{proof}

\section{Proof for odd degrees}
We finish this paper by proving the remaining part of Theorem \ref{maintheorem}, namely that $Z_\mathcal{L}(2N+1)\to\infty$ as $N\to \infty$. Let us briefly recall the set up of our argument from Sect. 3. We assume for a contradiction that there is a fixed integer $d$ and an infinite set $\mathcal{N}\subset\mathbf{N}$ of odd integers such that for each $N\in\mathcal{N}$ there exists a reciprocal Littlewood polynomial $P_{N}$ of the form \eqref{P_Ndefi} and degree $N$ which has at most $d$ unimodular roots. The corresponding polynomials $Q_{N}=\sum_{n=0}^{N}q^{(N)}_nz^n$ that we defined in \eqref{Qdefiodd} have coefficients \begin{align}
q^{(N)}_n = \begin{cases}
      \pm 1 & \text{for $n$ odd} \\
      0 & \text{for $n$ even.}
    \end{cases}
\label{q_nodd}
\end{align} and recall that for odd $N$ we have the relation \eqref{QtoPodd} between $P_{N}$ and $Q_N$. As we are assuming that each $P_{N}$ has at most $d$ unimodular roots, we deduce that $\Re(Q_N(e^{i\theta}))$ has at most $2d$ roots in $[0,2\pi]$ for each $N\in\mathcal{N}$ so that we may apply Lemma \ref{structurelem}. We obtain an infinite subset $\mathcal{N}'\subset \mathcal{N}$, constants $D,l=O_d(1)$ and $0=\rho_0<\dots<\rho_l=1$, and a sequence 
\begin{equation}    \left(\varepsilon^{(j)}_m\right)_{m=1}^D\in\{-1,0,1\}^D
\label{epsilonvaluesodd}
\end{equation} for each $j\in[l]$, so that for each $N\in\mathcal{N}'$:
\begin{equation}
    Q_{N}(z) = \sum_{j=1}^{l}\left(\sum_{n\in \big(\rho_{j-1}N,\rho_jN\big)}\varepsilon^{(j)}_{n \,(\!\!\!\!\!\!\mod D)}z^n\right)+ E_N(z),
\label{firstsimplodd}
\end{equation}
where the $E_N(z)$ are polynomials whose coefficients are $O(1)$ and with $o(N)$ terms. Note from the definition of the $\varepsilon^{(j)}_m$ in \eqref{epsilondefilemma} and from \eqref{q_nodd} that the period $D$ is even and that for each $j$:
\begin{equation}
    \varepsilon^{(j)}_m = \begin{cases}
      \pm 1 & \text{for $m$ odd} \\
      0 & \text{for $m$ even.}
    \end{cases}
\label{oddepsilonvalue}
\end{equation}Let us again define for each large $N\in\mathbf{N}$ the integer $r^{(N)}_j$ to be the largest multiple of $D$ in $\big(\rho_{j-1}N,\rho_jN\big)$ for $j\in[l]$ , and let $r^{(N)}_0=0$. Clearly $r^{(N)}_j=\rho_jN+O(1)$, so we can sum the geometric series in \eqref{firstsimplodd} to get that for each large $N\in\mathcal{N}'$:
\begin{align}
    Q_{N}(z) &= \frac{1}{z^D-1}\sum_{j=1}^{l}(z^{r^{(N)}_j}-z^{r^{(N)}_{j-1}})\left(\sum_{m=0}^{D-1}\varepsilon^{(j)}_{m}z^m\right)+ E'_N(z),
\nonumber\\
    &=\frac{1}{z^D-1}\sum_{j=1}^{l}A_j(z)(z^{r^{(N)}_j}-z^{r^{(N)}_{j-1}})+ E'_N(z), \label{furthersimpl+odd}
\end{align}
where the $E'_N(z)$ are polynomials whose coefficients are $O(1)$ and with $o(N)$ terms, and where for $1\leqslant j \leqslant l$ we defined
\begin{equation}
    A_j(z) = \sum_{m=0}^{D-1}\varepsilon^{(j)}_mz^j.
    \label{A_jdefinitionsodd}
\end{equation} We show that the $Q_{N}$ with $N\in\mathcal{N}'$ either satisfy the conditions of Proposition \ref{prop1} so that we may conclude as in the proof for even degrees, or else that all of the sequences $(\varepsilon^{(j)}_m)$ with $j\in[l]$ are identical. By \eqref{furthersimpl+odd}, we have for $N\in\mathcal{N}'$ that
\begin{align}
    \Re\left(Q_{N}(e(\theta))\right) &= \Re\left(\frac{1}{e(\theta D)-1}\sum_{j=1}^{l}A_j(e(\theta))(e(\theta r^{(N)}_j)-e(\theta r^{(N)}_{j-1}))\right) +E'_N(\theta),
\label{Noddlargescale}
\end{align}
where $||E'_N||_\infty =o_{N\to\infty}(N)$. Parseval's identity and the definition \eqref{A_jdefinitionsodd} of the $A_j$ give
\begin{align*}
    \sum_{m=0}^{D-1}\left|A_l(\zeta^{m})\right|^2&=D\sum_{m=0}^{D-1}\left|\varepsilon^{(l)}_m\right|^2 = \frac{D^2}{2},
\end{align*}
as $\varepsilon_m^{(j)}=\pm 1$ for $m$ odd and $\varepsilon_m^{(j)}=0$ for $m$ even by \eqref{oddepsilonvalue}. Similarly,
\begin{align}
    \sum_{m=0}^{D-1}\left|\Im\left(A_1(\zeta^{m})\right)\right|^2 
    &= D\sum_{m=1}^{D-1}\left(\frac{\varepsilon^{(1)}_{m}-\varepsilon^{(1)}_{D-m}}{2}\right)^2\leqslant \frac{D^2}{2}.
\label{e_iimaginary}
\end{align}
Hence, there either exists some $m$ for which $\left|A_l(\zeta^{m})\right|>\left|\Im\left(A_1(\zeta^{m})\right)\right|$ so that the condition \eqref{prop1cond} of Proposition \ref{prop1} holds with strict inequality, or else equality must hold in \eqref{e_iimaginary} and condition \eqref{prop1cond} for all $m$. Now, if some two of the sequences $(\varepsilon^{(j)}_m)_{m=1}^{D}, (\varepsilon^{(j+1)}_m)_{m=1}^{D}$ are distinct, then \eqref{prop1cond2} is satisfied for some $m$. We conclude that unless the sequences $(\varepsilon_m^{(j)})$ are identical for all $j$, then the conditions of Proposition \ref{prop1} are satisfied and we deduce that \eqref{Noddlargescale} has more than $2d$ roots for large $N\in\mathcal{N}'$, which is the required contradiction. 

\bigskip

To finish our proof, it therefore only remains to show that the following assumptions lead to a contradiction. Assume that $\mathcal{N}\subset \mathbf{N}$ is an infinite subset and that $\{Q_N:N\in\mathcal{N}\}$ is a collection of polynomials satisfying the conclusion of Lemma \ref{structurelem} with all of the sequences $(\varepsilon^{(j)}_m)_{m=1}^D$ being identical, and such that $\Re(Q_N(e^{i\theta}))$ has at most $2d$ roots in $[0,2\pi]$ for each $N\in\mathcal{N}$. Recall from Sect. 3 that there do exist functions $h_m(\theta)=\Re\left(\sum_{n=0}^{2m}(-1)^{n+1}e^{i(2n+1)\theta}\right)$ with a periodic sequence of coefficients that have a bounded number of sign changes on $[0,2\pi]$ for all $m$. This shows that one cannot establish the conclusion of Proposition \ref{prop1} in this case, and hence we need to proceed by a different argument. Note that our assumptions in this final case of the argument correspond precisely to the assumptions of Corollary \ref{coroflem1} and we obtain the following from its conclusion. There exist constants $L=O_d(1)$ and $0\leqslant\tau_0\leqslant\dots\leqslant\tau_{L}\leqslant 1$, sequences $(\delta^{(j)}_m)_{m=1}^D$ for $j\in[L]$, and for each $N\in\mathcal{N}$ a collection of $L$ disjoint intervals $J^{(N)}_j\subset [N]$, $j\in [L]$, such that the following holds. 
\begin{enumerate}
    \item[(i)] The starting point and the length of each of the intervals $J^{(N)}_j$ is a multiple of $D$, and $\min J^{(N)}_j,\max J^{(N)}_j=\tau_j N+o(N)$ for each $j\in[L]$. In particular, the intervals have length $|J^{(N)}_j|=o_{N\to\infty,N\in\mathcal{N}}(N)$.
    \item[(ii)] For each $N\in\mathcal{N}$, we can write
\begin{equation}
    Q_{N}(z)= \left(\sum_{m=0}^{D-1} \varepsilon_mz^m\right)\frac{z^{N+1}-1}{z^D-1}+\sum_{j=1}^{L}\left(\sum_{m\in J^{(N)}_j}\delta^{(j)}_{m\,(\!\!\!\!\!\!\mod D)}z^m\right).
\label{Q_Ngeom}
\end{equation}
\end{enumerate}
We collect some important properties of the sequence $(\varepsilon_m)_{m=0}^{D-1}$. Recall that we may assume that equality holds in \eqref{e_iimaginary} and hence we deduce that the sequence $(\varepsilon_m)_{m=0}^{D-1}$ satisfies
\begin{equation}
    \varepsilon_m=-\varepsilon_{D-m},
\label{epsilonoddfunc}
\end{equation}
and from \eqref{oddepsilonvalue} we see that
\begin{equation}
    \varepsilon_m = \begin{cases}
      \pm 1 & \text{for $m$ odd} \\
      0 & \text{for $m$ even.}
    \end{cases}
\label{finalepsilonvalue}
\end{equation}
By comparing coefficients of the polynomials on both sides of \eqref{Q_Ngeom} and by using \eqref{q_nodd} and \eqref{finalepsilonvalue}, we obtain the following useful description of the sequences $(\delta^{(j)}_m)_{m=0}^{D-1}$ for $j\in[L]$:
\begin{equation}
    \delta^{(j)}_m\in\{-2\varepsilon_m,0\}.  
    \label{deltadefi}
\end{equation}
To get a contradiction, our goal is to show that $\Re(Q_{N}(e^{i\theta}))$ has more than $2d$ roots for some $N\in\mathcal{N}$ and we begin by rewriting \eqref{Q_Ngeom} in the form
\begin{align}
    \Re(Q_{N}(e^{i\theta}))
    = \frac{1}{2\sin(D\theta)}\Im\left(A_0(e^{i\theta})\left(e^{iN\theta}-1\right)+\sum_{j=1}^{L}A^{(N)}_j(e^{i\theta})\right)
\label{1patternim}
\end{align}
where \begin{equation}
    A_0(z)\vcentcolon=\sum_{m=0}^{D-1} \varepsilon_mz^{m-D/2}\label{A_0defi}
\end{equation} is a fixed Laurent polynomial independent of $N$, and where for $j\in [L]$ we write
\begin{equation}
    A^{(N)}_j(z)\vcentcolon=\left(z^{\max J^{(N)}_j+1}-z^{\min J^{(N)}_j}\right)\sum_{m=0}^{D-1} \delta^{(j)}_m z^{m-D/2}.
\label{A_jdefi}
\end{equation}
We record some useful properties of the polynomials $A^{(N)}_j$. Observe first that for each $j\in [L]$, all terms in $A^{(N)}_j(z)$ have degree $\tau_j N+o(N)$ as $\min J^{(N)}_j,\max J^{(N)}_j=\tau_jN+o(N)$ by item (i) above  and because $D=O(1)$, and further observe that each $A^{(N)}_j$ consists of at most $2D=O(1)$ such terms. We remind the reader that for the rest of the paper, unless explicitly stated otherwise, the asymptotic notation describes the behaviour as $N\to\infty$. Our next step is to group together the terms in \eqref{1patternim} whose degrees have asymptotically the same size. Hence, let $\rho_1,\dots,\rho_{l-1}$ be the distinct values in the sequence $\tau_0,\dots,\tau_{L}$ other than $0,1$ and let $\rho_0=0,\rho_{l}=1$. After combining terms in \eqref{1patternim} with equal $\tau_j$, it clearly suffices to show that the following function has more than $2d+2D$ roots for some $N\in\mathcal{N}$:
\begin{align}
    H_N(\theta) =\Im\left(\sum_{j=0}^{l}B^{(N)}_j(\theta)e(\rho_jN\theta)\right),  
\label{H_kwithB}
\end{align}
where 
\begin{align}
    B^{(N)}_j(\theta)=\begin{cases}
    A_0(e(\theta)) +e(-N\theta)\sum_{m:\tau_m=1}A^{(N)}_m(e(\theta))&\text{if } j=l,\\
    e(-\rho_jN\theta)\sum_{j:\tau_m=\rho_j}A^{(N)}_m(e(\theta)) &\text{if } 1\leqslant j\leqslant l-1,\\
    -A_0(e(\theta))+\sum_{m:\tau_m=0}A^{(N)}_m(e(\theta)) &\text{if } j=0.
    \end{cases}
\label{B_jdefi}
\end{align}
We see from \eqref{A_jdefi} that each of the $B^{(N)}_j(\theta)$ has at most $2D(L+1)=O(1)$ terms and since we noted that each term in $A^{(N)}_j$ has degree $\tau_j N+o(N)$, we obtain the crucial fact that the $B^{(N)}_j$ have degree $o(N)$. Having written $H_N$ in this way, it is clear from the properties of the $B^{(N)}_j$ that we just stated that the conditions of Lemma \ref{lem2} are satisfied. By applying Lemma \ref{lem2} with a sequence $(\gamma_N)$ with $\gamma_N\in [0,1]$ and a constant $C>0$, we obtain an infinite subset $\mathcal{N}'\subset \mathcal{N}$ and constants $b_j\in\mathbf{C}$ so that for $u\in(-C,C)$ and $N\in\mathcal{N}'$: 
\begin{equation}
    H_N\left(\gamma_N+\frac{u}{N}\right)= \Im\left(\sum_{j=0}^{l}b_je(\rho_j u)\right)+o_{C;N\to\infty}(1).
\label{H_klocal}
\end{equation}
Note the important property that the $b_j$ are constants independent of $N$ and hence the first function on the right hand side is a fixed trigonometric function independent of $N$. Our strategy is to find a good choice of the sequence $\gamma_N$ so that this function will have many sign changes on $(-C,C)$, and in practice this will involve us finding $(\gamma_N)$ such that the coefficients $b_j$ satisfy the assumptions of Lemma \ref{signchange}. To continue, it is convenient to first handle the special case for which $B^{(N)}_l(\theta)=-B^{(N)}_0(\theta)$ in a separate lemma.
\begin{lemma} \label{lemtrivial}
Let $\mathcal{N}\subset \mathbf{N}$ be an infinite set, let $R\subset\mathbf{Z}$ be finite and let $l,D\in\mathbf{N}$. Suppose that $G_N(\theta),N\in\mathcal{N}$ are functions, that $0= \rho_0<\dots<\rho_l= 1$ and that for each $N\in\mathcal{N}$ there are $l+1$ trigonometric functions $B^{(N)}_j(\theta)$ with coefficients in $R$, with at most $D$ non-zero coefficients, and each having degree $\deg B^{(N)}_j = o(N)$ so that we can write for $N\in\mathcal{N}$:
    \begin{equation}
        G_N(\theta) = \Im\left(\sum_{j=0}^lB^{(N)}_j(\theta)e(\rho_jN\theta)\right).
    \label{G_Nlemmadefi}
    \end{equation} 
    If $B^{(N)}_{l}(\theta)=-B^{(N)}_0(\theta)$ are identical as functions of $\theta$ for infinitely many $N\in\mathcal{N}$, then for each constant $C$ there exists a function $G_N(\theta)$ having at least $C$ zeros.
\end{lemma}
\begin{proof}
    We split the argument into two cases. In the first case, we assume that there is an infinite set $\mathcal{M}\subset\mathcal{N}$ so that $B^{(N)}_{l}(\theta)=-B^{(N)}_0(\theta)$ for all $N\in\mathcal{M}$ and in addition that each of the other trigonometric polynomials $B^{(N)}_j$ with $1\leqslant j\leqslant l-1$ is identically $0$. Then from \eqref{G_Nlemmadefi} we get that $G_N(\theta)= \Im\left(-B^{(N)}_{0}(\theta)e(N\theta)+B^{(N)}_0(\theta)\right)$ and it is clear that $\theta=n/N$ is a root for all $n\in[N]$ so that we obtain arbitrarily many roots for large $N\in\mathcal{M}$. In the second case, we may assume that there is an infinite set $\mathcal{M}\subset\mathcal{N}$ and an index $j\in[l-1]$ so that $B^{(N)}_{l}(\theta)=-B^{(N)}_0(\theta)$ and that $B^{(N)}_j$ is not identically $0$ for any $N\in\mathcal{M}$. Note that for every non-zero $B^{(N)}_j$ we have $\int_0^{1}|B^{(N)}_j(\theta)|^2\,d\theta \geqslant 1$ by Parseval's Theorem and as all its Fourier coefficients are in $R\subset \mathbf{Z}$ by assumption. It follows that for each $N\in\mathcal{M}$ there exists a $\gamma_N\in[0,1]$ so that $|B^{(N)}_j(\gamma_N)|\geqslant 1$. The assumptions of Lemma \ref{lemtrivial} imply all of the assumptions of Lemma \ref{lem2} so we may apply it with the sequence $(\gamma_N)_{N\in\mathcal{M}}$ to obtain the following. There exists an infinite subset $\mathcal{N}'\subset\mathcal{M}$ and constants $b_m\in\mathbf{C}$ for $m\in\{0,1,\dots,l\}$ such that for $u\in(-C,C)$ and $N\in\mathcal{N}'$:
    \begin{equation}
    G_N\left(\gamma_N+\frac{u}{N}\right) = \Im\left(\sum_{m=0}^{l}b_me(\rho_m u)\right)+o_{C;N\to\infty}(1)
\label{lem4local}
    \end{equation}
and note that by \eqref{b_jlem2size} we can determine the size of the constants $b_m$ using
\begin{equation*}
\lim_{N\to\infty,N\in\mathcal{N}'}\left|B^{(N)}_m(\gamma_N)\right|=|b_m|.
\end{equation*}
In particular, as $B^{(N)}_l(\theta)=-B^{(N)}_0(\theta)$ by assumption we deduce that $|b_l|\geqslant|b_0|$ and as we chose the $\gamma_N$ so that $|B^{(N)}_j(\gamma_N)|\geqslant 1$ we also have that $b_j\neq 0$. Hence, we can apply Lemma \ref{signchange} to deduce that the first function on the right hand side of \eqref{lem4local} (which now is a fixed function of $u$, independent of $N$) has infinitely many sign changes on $\mathbf{R}$. If we choose $C$ to be a sufficiently large constant and then let $N\to\infty,N\in\mathcal{N}'$ it follows that $G_N$ has arbitrarily many roots for large $N$, as desired.
\end{proof}

To deal with the remaining and more substantial case in which $B^{(N)}_l(\theta)$ and $-B^{(N)}_0(\theta)$ are distinct functions of $\theta$ for all large $N\in\mathcal{N}$, we establish the following proposition.
\begin{proposition}
\label{prop2}
    Let $\mathcal{N}\subset\mathbf{N}$ be an infinite set, $D\in2\mathbf{N}$ and $C>0$. Suppose that $(\varepsilon_m)\in\{-1,0,1\}^D$ is a sequence which is not identically zero that satisfies $\varepsilon_m=-\varepsilon_{D-m}$, and define the sine polynomial \begin{equation}
    \label{a(t)defi}
        a(t)=i^{-1}\sum_{m=-D/2}^{D/2}\varepsilon_me(mt)=2\sum_{m=0}^{D/2}\varepsilon_m\sin(mt).
    \end{equation} Suppose that for each $N\in\mathcal{N}$ we have two sine polynomials that are not both identically zero:
    \begin{align}
        s_1^{(N)}(t)&= \sum_{m\in \Lambda_1^{(N)}} \delta_m\left(\sin mt-\sin((m-p_m^{(N)}D)t)\right), \label{s_jdefiinprop}\\
    s_2^{(N)}(t)&= \sum_{m\in \Lambda_2^{(N)}} \delta_m\left(\sin mt-\sin((m-p_m^{(N)}D)t)\right),\nonumber
 \end{align}
where $\Lambda_j^{(N)}$ are sets of positive integers of size at most $C$, where $\delta:\mathbf{N}\to\mathbf{Z}$ is given by
\begin{equation}
    \delta_m= -2\varepsilon_{m(\!\!\!\!\mod D)}
    \label{deltafound}
\end{equation} and where the $p_m^{(N)}$ are positive integers\footnote{It is a slight yet unimportant abuse of notation that we use $p_m^{(N)}$ in the expressions of both $s_1^{(N)}$ and $s_2^{(N)}$ since technically $p_m^{(N)}$ may also depend on whether $m\in\Lambda^{(N)}_1$ or $m\in \Lambda^{(N)}_2$.} satisfying $|m|\geqslant|m-p_m^{(N)}D|$ for each $m\in\Lambda^{(N)}_j$. Then there exists a $\kappa>0$ and an infinite subset $\mathcal{N}'\subset\mathcal{N}$ such that for each $N\in\mathcal{N}'$ we can find a $\gamma_N\in[0,1]$ satisfying 
\begin{equation}
    \left|a(\gamma_N)+s_{1}^{(N)}(\gamma_N)\right|\geqslant \left|a(\gamma_N)-s_{2}^{(N)}(\gamma_N)\right|+\kappa.
    \label{finaltaskprop}
\end{equation}
\end{proposition}
Before proving Proposition \ref{prop2}, let us show how it may be used to finish the proof. Let $\mathcal{N}\subset\mathbf{N}$ be the infinite set such that the functions $H_N$ in \eqref{H_kwithB} have at most $2d+2D$ roots for $N\in\mathcal{N}$. From Lemma \ref{lemtrivial} we deduce that $B^{(N)}_l(\theta)$ and $-B^{(N)}_0(\theta)$ are distinct functions of $\theta$ for all large $N\in\mathcal{N}$. We begin by checking that, for each large $N\in\mathcal{N}$,  we can write the functions $B^{(N)}_0,B^{(N)}_l$ from \eqref{B_jdefi} in the form
\begin{align*}
    \Im\left(B^{(N)}_l(\theta)\right)&=a(\theta)+s^{(N)}_1(\theta)\\
    \Im\left( B^{(N)}_0(\theta)\right)&=-a(\theta)+s^{(N)}_2(\theta),
\end{align*}
where $a(\theta)$ and $s^{(N)}_1(\theta),s^{(N)}_2(\theta)$ with $N\in\mathcal{N}$ are functions that satisfy the conditions of Proposition \ref{prop2}. From the definition of the $B^{(N)}_j$ in \eqref{B_jdefi} we see that
\begin{align*}
    B^{(N)}_l(\theta)&=A_0(e(\theta))+e(-N\theta)\sum_{j:\tau_j=1}A^{(N)}_j(e(\theta)),\\
    B^{(N)}_0(\theta)&=-A_0(e(\theta))+\sum_{j:\tau_j=0}A^{(N)}_j(e(\theta)).
\end{align*}
The definition \eqref{A_0defi} of $A_0$ and the fact that $\varepsilon_m=-\varepsilon_{D-m}$ by \eqref{epsilonoddfunc} imply that $\Im(A_0)=i^{-1}\sum_{m=-D/2}^{D/2}\varepsilon_me(mt)$, so $a(t)\vcentcolon=\Im(A_0(e(t))$ does indeed satisfy the assumptions of Proposition \ref{prop2}. We then take
\begin{align*}
    s^{(N)}_1(t)&=\Im \sum_{j:\tau_j=1}A^{(N)}_j(e(t))e(-Nt),\\
    s^{(N)}_2(t)&=\Im\sum_{j:\tau_j=0}A^{(N)}_j(e(t)).
\end{align*}
Recall the definition \eqref{A_jdefi} of the polynomials $A^{(N)}_j$ so as they have at most $2D=O(1)$ terms, as their coefficients satisfy \eqref{deltadefi} and as the lengths of the intervals $J^{(N)}_j$ are multiples of $D$ by item (i) on page \pageref{Q_Ngeom}, we see that $s^{(N)}_1,s^{(N)}_2$ consist of $O(1)$ terms of the form $-\delta_m(\sin mt-\sin(m-pD)t)$ where $|m|\geqslant |m-pD|$ and hence that they are sine polynomials of the desired form \eqref{s_jdefiinprop}. Finally, the functions $s^{(N)}_1,s^{(N)}_2$ are not both identically zero for any large
$N\in\mathcal{N}$ since we deduced from Lemma \ref{lemtrivial} that $B^{(N)}_l\neq -B^{(N)}_0$ for all large $N\in\mathcal{N}$. Hence, we may apply Proposition \ref{prop2} to deduce that there exists an infinite set $\mathcal{N}'\subset\mathcal{N}$ and a sequence $(\gamma_N)_{N\in\mathcal{N}'}$ such that \eqref{finaltaskprop} holds, which in this setting means that
\begin{equation}
\label{largevaluesB_0B_l} \left|\Im(B^{(N)}_l(\gamma_N))\right|\geqslant\left|\Im(B^{(N)}_0(\gamma_N))\right|+\kappa.
\end{equation} Similarly as in \eqref{H_klocal}, we may then apply Lemma \ref{lem2} with the functions $H_N(\theta),N\in\mathcal{N}'$ and the sequence $(\gamma_N)_{N\in\mathcal{N}'}$ provided by Proposition \ref{prop2} to obtain an infinite subset $\mathcal{N}''\subset\mathcal{N}'$ and constants $b_m\in\mathbf{C}$ for $m\in\{0,1,\dots,l\}$ such that for $u\in(-C,C)$ and $N\in\mathcal{N}''$ the functions $H_N$ in \eqref{H_kwithB} satisfy:
    \begin{equation}
    H_N\left(\gamma_N+\frac{u}{N}\right) = \Im\left(\sum_{m=0}^{l}b_me(\rho_m u)\right)+o_{C;N\to\infty}(1)
\label{finallocalversion}
    \end{equation}
and furthermore by \eqref{b_jlem2size} that
\begin{equation*}
\lim_{N\to\infty,N\in\mathcal{N}'}B^{(N)}_m(\gamma_N)e(\rho_mN\gamma_N)=b_m.
\end{equation*} In particular, since \eqref{largevaluesB_0B_l} holds for all $N\in\mathcal{N}'$ and as $\rho_0=0$, we get that $|b_l|\geqslant|\Im(b_0)|+\kappa$. Hence, the first function on the right hand side of \eqref{finallocalversion} (which is a trigonometric function of $u$, independent of $N$) satisfies the conditions of Lemma \ref{signchange}, so we may apply Lemma \ref{signchange} to deduce that it has infinitely many sign changes on $\mathbf{R}$. If we choose $C$ to be a sufficiently large constant and then let $N\to\infty,N\in\mathcal{N}''$ it follows from \eqref{finallocalversion} that $H_N$ has arbitrarily many roots for large $N\in\mathcal{N}''$, which is the required contradiction. Note that to obtain the crucial inequality $|b_l|>|\Im(b_0)|$ which allows us to use Lemma \ref{signchange}, we require the constant $\kappa>0$ appearing in \eqref{largevaluesB_0B_l} to be independent of $N$. Showing that such a $\kappa>0$ which is independent of $N$ exists is the main difficulty in the proof Proposition \ref{prop2}. It seems possible that Proposition \ref{prop2} is in fact a special case of the following question to which the author does not know the answer.
\begin{prob}
    Is it true that for any positive integer $k\in\mathbf{N}$, there exists a $\kappa=\kappa(k)>0$ such that the following holds? Let $s_1(t),s_2(t)$ be two distinct sine polynomials of the form $s_j(t)=\sum_{n\in\Lambda_j}\varepsilon_{n,j}\sin nt$ with $|\Lambda_j|\leqslant k$ terms and coefficients $\varepsilon_{n,j}=\pm 1$, then there exists a $\gamma\in[0,1]$ such that $|s_1(\gamma)|\geqslant|s_2(\gamma)|+\kappa.$
\end{prob}
Returning to the main argument, it only remains to prove Proposition \ref{prop2}. 
\begin{proof}[Proof of Proposition \ref{prop2}]
Our task consists of finding a constant $\kappa>0$ and an infinite subset $\mathcal{N}'\subset\mathcal{N}$ such that for each $N\in\mathcal{N}'$, there exists some $t=t(N)$ that satisfies
\begin{equation}
    \left|a(t)+s_{1}^{(N)}(t)\right|\geqslant \left|a(t)-s_{2}^{(N)}(t)\right|+\kappa.
    \label{finaltaskinproof}
\end{equation}
Let us define the function $p:\mathcal{N}\to\mathbf{N}$ by
\begin{equation}
    p=p(N)=\max\{p_m^{(N)}:m\in\Lambda_1^{(N)}\cup\Lambda_2^{(N)}\}.
    \label{pdefi}
\end{equation} As the sets $\Lambda_j^{(N)}$ have size at most $C$ for all $N\in\mathcal{N}$ by the assumptions of Proposition \ref{prop2}, we can use a compactness argument to find an infinite subset $\mathcal{N}_1\subset\mathcal{N}$ for which $\Lambda_1^{(N)},\Lambda_2^{(N)}$ have fixed size and such that all of the limits $\lim_{N\to\infty,N\in\mathcal{N}_1} \frac{p_m^{(N)}}{p(N)}$ converge to non-negative constants for $m\in\Lambda_1^{(N)}\cup\Lambda_2^{(N)}$. It will be technically convenient for later to get rid of the contribution from those $m\in \Lambda_1^{(N)}\cup\Lambda_2^{(N)}$ for which $p_m^{(N)}=o_{N\to\infty}(p)$, so let us define 
\begin{equation}
\label{Lambda^Ndefi}
    \Lambda^{(N)}\vcentcolon= \left\{m\in\Lambda_1^{(N)}\cup\Lambda_2^{(N)}:\lim_{N\to\infty,N\in\mathcal{N}_1}\frac{p_m^{(N)}}{p(N)}>0\right\}.
\end{equation} Throughout the proof, the limit $N\to\infty$ should always be taken in an appropriate infinite subset $\mathcal{N}_i\subset\mathcal{N}$ but we will not explicitly state this every time. We define the following truncated version of the functions in \eqref{s_jdefiinprop}:
\begin{align}
    s_\Lambda^{(N)}(t) \vcentcolon= &\sum_{m\in \Lambda_1^{(N)}\cap\Lambda^{(N)}} \delta_m\left(\sin mt-\sin((m-p_m^{(N)}D)t)\right)\label{s_Lambda}\\+&\sum_{m\in \Lambda_2^{(N)}\cap\Lambda^{(N)}} \delta_m\left(\sin mt-\sin((m-p_m^{(N)}D)t)\right)\nonumber
\end{align}
and let us call the first term on the right hand side $s_{1,\Lambda}^{(N)}(t)$ and the second $s_{2,\Lambda}^{(N)}(t)$. To find a $t$ for which \eqref{finaltaskinproof} holds, we will consider angles $t=\frac{2\pi r}{D}+\frac{u}{p(N)}$ where $u$ ranges over $(-c,c)$ for a constant $c$ to be chosen later. We will need the following lemma later to show that we may work with the functions $s^{(N)}_{j,\Lambda}(t)$ instead of $s^{(N)}_{j}(t)$.
\begin{lemma}
\label{s_jsmallneargamma}
Let $s^{(N)}_1,s^{(N)}_2$ be the sine polynomials in \eqref{s_jdefiinprop} and let $s^{(N)}_{1,\Lambda},s^{(N)}_{2,\Lambda}$ be their truncations that we defined in \eqref{s_Lambda}. Write $\gamma_r=\frac{2\pi r}{D}$ and define for $c>0$ the intervals $I_N(c)\vcentcolon=\left(\gamma_r-\frac{c}{p(N)},\gamma_r+\frac{c}{p(N)}\right)$ for $N\in\mathcal{N}_1$. Then for any given $\rho>0$, one can choose $c>0$ sufficiently small such that for all $N\in\mathcal{N}_1$ we have
\begin{equation}
\label{s_jsmall}
    \sup_{t\in I_N(c)}\left|s_j^{(N)}\left(t\right)\right|<\rho,
\end{equation}
for $j=1,2$. Furthermore, we have for all $t\in I_N(c)$ that 
\begin{align}
\label{lemsmallpart2}
    s_{j}^{(N)} \left(t\right)= s_{j,\Lambda}^{(N)}\left(t\right)+o_{c;N\to\infty}(1).
\end{align}
\end{lemma} 
\begin{proof}
We use the bound $|\sin \alpha-\sin\beta|\ll |\alpha-\beta|$ to deduce for $t=\frac{2\pi r}{D}+\frac{u}{p}$ with $u\in(-c,c)$ that 
\begin{align}
\label{ineqinlemms_j}
    \left|\sin mt-\sin((m-p_m^{(N)}D)t)\right|\ll p_m^{(N)} |u|/p(N)\leqslant c,
\end{align} 
where we used the definition \eqref{pdefi} of $p=p(N)$ for the final inequality. Since each sine polynomial $s^{(N)}_j$ is a sum of at most $C$ terms of the form $\delta_m(\sin mt-\sin((m-p_m^{(N)}D)t))$ by the assumption of Proposition \ref{prop2}, this inequality implies that \eqref{s_jsmall} holds when $c>0$ is chosen to be sufficiently small in terms of $\rho$. Further, since we have that $p_m^{(N)}=o(p(N))$ by definition \eqref{Lambda^Ndefi} when $m\notin \Lambda^{(N)}$, the second part of Lemma \ref{s_jsmallneargamma} follows immediately from the first inequality in \eqref{ineqinlemms_j}.
\end{proof}

The next step in our approach is to find a useful decomposition of $s_\Lambda^{(N)}$. Recall that $\Lambda^{(N)}$ is a finite set of fixed size $\lambda$ for all $N\in\mathcal{N}_1$ and let us order it as $\Lambda^{(N)}=\{m_1^{(N)}<\dots<m_{\lambda}^{(N)}\}$. We may then restrict to an infinite subset $\mathcal{N}_2\subset\mathcal{N}_1$ such that each of the limits \begin{equation*}
     \lim_{N\to\infty,N\in\mathcal{N}_2}\frac{m_{j}^{(N)}-m_{j-1}^{(N)}}{p(N)}
\end{equation*} converges for $2\leqslant j\leqslant \lambda$, either to a non-negative constant or to $+\infty$. Our aim is to partition the set $\Lambda^{(N)}$ into maximal sets whose elements all differ by $O(p(N))$. To achieve this, let $m_1^*(N)<\dots<m_{\lambda'}^*(N)$ be all of the $m_j^{(N)}\in\Lambda^{(N)}$ for which $\lim_{N\to\infty,N\in\mathcal{N}_2}\frac{m_{j}^{(N)}-m_{j-1}^{(N)}}{p}=\infty$, and we then partition $\Lambda^{(N)}$ into the sets \begin{equation}
\label{firstpartdefi}
     \Lambda^{(N,j)}\vcentcolon=\{m\in\Lambda^{(N)}:m_j^*(N)\leqslant m<m_{j+1}^*(N)\}.
\end{equation} 
Note that these $\Lambda^{(N,j)}$ indeed satisfy the desired property that $m-m'=O(p(N))$ if $m,m'\in\Lambda^{(N,j)}$ are in the same part of the partition, and that $|m-m'|/p(N)\to\infty$ if $m\in\Lambda^{(N,j)}$ and $m'\in\Lambda^{(N,j')}$ with $j\neq j'$. Hence, for each $N\in\mathcal{N}_2$ we may equivalently define the parts of the partition by
\begin{equation}
     \Lambda^{(N,j)}\vcentcolon=\{m\in\Lambda^{(N)}:|m-m_j^*(N)|=O(p)\}.
\label{Lambdapart}
\end{equation} 
We may now obtain the useful decomposition of the function $s_\Lambda^{(N)}$ in \eqref{s_Lambda} by combining the terms coming from those $m$ in the same $\Lambda^{(N,j)}$ for each $j$, $1\leqslant j\leqslant \lambda'$. First however, we note that it will be a technical convenience to use the integers $n_j(N)$ instead of the $m_j^*(N)$ in the decomposition, where we define $n_j(N)$ to be the largest multiple of $D$ less than $m_j^*(N)-p(N)D$. We can now write down the desired decomposition of the functions $s^{(N)}_\Lambda$ for each $N\in\mathcal{N}_2$ and we collect its useful properties in the following lemma.
\begin{lemma}
\label{s_lambdadecolem}
        For each $N\in\mathcal{N}_2$, we can write the sine polynomial $s^{(N)}_\Lambda$ from \eqref{s_Lambda} in the form
\begin{equation}
    s_\Lambda^{(N)}(t)=\Im\left(\sum_{j=1}^{\lambda'} C_j^{(N)}(t)e^{in_j(N) t}\right)
\label{s_lambdadecomposition}
\end{equation}
where the functions $C_j^{(N)}(t)$ are trigonometric polynomials of the form
\begin{equation}
    C_j^{(N)}(t)=e^{-in_j(N)t}\sum_{m\in\Lambda^{(N,j)}}\delta_m\left(e^{imt}-e^{i(m-p_m^{(N)}D)t}\right).
\label{C_jdefi}
\end{equation}
Moreover, the decomposition \eqref{s_lambdadecomposition} has the following properties.
\begin{enumerate}
    \item[(i)] For each $N\in\mathcal{N}_2$, we have that
\begin{equation}
    n_j(N)= m^*_j(N)-p(N)D+O(1)
\label{n_jsize}
\end{equation} and that the partition $\Lambda^{(N)}=\cup_{j=1}^{\lambda'}\Lambda^{(N,j)}$ from \eqref{Lambdapart} can equivalently be defined by
\begin{equation}
    \Lambda^{(N,j)}=\{m\in\Lambda^{(N)}: |m-n_j(N)| = O(p(N))\},
\label{newpartdefi}
\end{equation} for $1\leqslant j\leqslant \lambda'$.
\item[(ii)] For each $N\in\mathcal{N}_2$, the $C_j^{(N)}$ are trigonometric polynomials of degree $O(p)$ with $O(1)$ terms for $j\in[\lambda']$.
\item[(iii)] The ratios between the factors $e^{n_j(N)t}$ oscillate on a much faster scale than the trigonometric polynomials $C^{(N)}_j(t)$, by which we mean that \begin{equation}
    \lim_{N\to\infty,N\in\mathcal{N}_2}\frac{n_{j+1}(N)-n_{j}(N)}{p(N)}=+\infty
\label{n_jscales}
\end{equation} for all $j\in[\lambda'-1]$.
\end{enumerate}
\end{lemma}
\begin{proof}
    The proof is merely a matter of unpacking definitions. The fact that we may write the sine polynomial $s^{(N)}_\Lambda$ from \eqref{s_Lambda} in the form \eqref{s_lambdadecomposition} is immediate from the definitions \eqref{C_jdefi} of the $C^{(N)}_j$ and the fact that $\Lambda^{(N)}=\cup_{j=1}^{\lambda'}\Lambda^{(N,j)}$. For item (i), recall that just above the statement of Lemma \ref{s_lambdadecolem}, we defined the integers $n_j(N)$ for $1\leqslant j\leqslant\lambda'$ to be the largest multiple of $D$ less than $m^*_j(N)-p(N)D$ so that \eqref{n_jsize} is clear. This also shows that $n_j(N)=m^*(N)+O(p(N))$ so that the observation that the partition of $\Lambda^{(N)}$ may be defined by \eqref{newpartdefi} is immediate from \eqref{Lambdapart}. For item (ii), we recall that $\mathcal{N}_2$ was in part defined so that $|\Lambda^{(N)}|=\lambda$ for all $N\in\mathcal{N}_2$ so it follows that each $C^{(N)}_j$ is a trigonometric polynomial with at most $\lambda=O_{N\to\infty}(1)$ terms. We further observe from its definition \eqref{C_jdefi} that all frequencies of $C^{(N)}_j$ are of the form $m-n_j(N)$ or $m-p^{(N)}_mD-n_j(N)$ for some $m\in\Lambda^{(N,j)}$. Hence, from \eqref{newpartdefi} and as $p^{(N)}_m=O(p(N)))$ holds trivially by the definition \eqref{pdefi} of $p$, we see that $\deg C^{(N)}_j(t)=O(p(N))$. Finally, item (iii) follows from the definitions of the integers $m^*_j(N)$ just before \eqref{firstpartdefi} which state that $|m^*_j(N)-m^*_{j-1}(N)|/p(N)\to\infty$ as $N\to\infty,N\in\mathcal{N}_2$ and the fact that $n_j(N)=m^*_j(N)+O(p)$ by \eqref{n_jsize}.
\end{proof}

Next, we analyse the trigonometric polynomials $C_j^{(N)}(t)$. Recall that we write $\gamma_r=\frac{2\pi r}{D}$ so from \eqref{C_jdefi} it is clear that the imaginary part of $C^{(N)}_j$ has the following second derivative at $\gamma_r$:
\begin{align*}
    \Im(C_j^{(N)})''\left(\gamma_r\right)&= \sum_{m\in \Lambda^{(N,j)}} -\delta_m\left((m-n_j(N))^2-(m-n_j(N)-p_m^{(N)}D)^2\right)\frac{\zeta^{mr}-\zeta^{-mr}}{2i}\\
    &=D\sum_{m\in \Lambda^{(N,j)}} -\delta_m p_m^{(N)}(2m-2n_j(N)-p_m^{(N)}D)\frac{\zeta^{mr}-\zeta^{-mr}}{2i}.
\end{align*} Note that by \eqref{n_jsize}, the integers $2m-2n_j(N)-p_m^{(N)}D$ appearing here are of magnitude $\Theta(p(N))$, for all $m\in\Lambda^{(N,j)}$. The function $a(t)$ defined in \eqref{a(t)defi} satisfies
\begin{align*}
    a(\gamma_r)=i^{-1}\sum_{m=-D/2}^{D/2}\varepsilon_m\zeta^{mr}=\sum_{m=-D/2}^{D/2}\varepsilon_m\frac{\zeta^{mr}-\zeta^{-mr}}{2i}.
\end{align*}
Since $\delta_m=-2\varepsilon_{m(\!\!\!\!\mod D)}$ and as $\varepsilon_m=-\varepsilon_{D-m}$ by the assumptions of Proposition \ref{prop2}, Parseval's identity then shows for each $j\in[\lambda']$ that
\begin{align}
\label{largederivative}
    &\sum_{r=0}^{D-1}a(\gamma_r)\Im(C_j^{(N)})''\left(\gamma_r\right)\\
    &=D^2\!\!\!\sum_{r\in\mathbf{Z}/D\mathbf{Z}}-\delta_m\varepsilon_m\!\!\left(\sum_{m\in\Lambda^{(N,j)}, m\equiv r(\!\!\!\!\!\mod D)}p_m^{(N)}(2m-2n_j(N)-p_m^{(N)}D)\!\right)\nonumber\\
    &=2D^2\!\!\!\sum_{m\in\Lambda^{(N,j)}}p_m^{(N)}(2m-2n_j(N)-p_m^{(N)}D)\!\nonumber\\
    &= \Theta(p(N)^2) \nonumber
\end{align}
where the lower bound in the final line follows as $p^{(N)}_m\gg p(N)$ for $m\in\Lambda^{(N)}$ by the definition \eqref{Lambda^Ndefi} of $\Lambda^{(N)}$ and from the fact that $2m-2n_j(N)-p_m^{(N)}D=\Theta(p)$ for $m\in\Lambda^{(N,j)}$ by \eqref{n_jsize}, and the corresponding upper bound from this same fact and as $\left|\Lambda^{(N,j)}\right|\leqslant\lambda=O(1)$. 

\medskip

On the other hand, note that we have the trivial upper bound $||\Im(C_j^{(N)})''||_\infty =O(p^2)$ as $C_j^{(N)}$ is a trigonometric polynomial of degree $O(p)$ with $O(1)$ terms by item (ii) in Lemma \ref{s_lambdadecolem}. We also have the trivial bound $||a(t)||_\infty=O(1)$ since the function $a(t)$ that we defined in \eqref{a(t)defi} is a fixed sine polynomial, independent of $N$. Combining these trivial upper bounds with the bound in \eqref{largederivative}, we therefore deduce that for each $N\in\mathcal{N}_2$, there exists some $r=r(N)$, $0\leqslant r\leqslant D-1$, such that \begin{align}
    |a(\gamma_r)|&\gg 1 \text{ and }\label{aCamplify}\\
    |\Im(C_{j}^{(N)})''(\gamma_r)|&\gg p(N)^2,\nonumber
\end{align} where the implied constants are independent of $N$, and moreover that $a(\gamma_r)$ and $\Im(C_{j}^{(N)})''(\gamma_r)$ have the same sign. Let us define for a constant $c>0$ the intervals 
\begin{align*}
    I_N(c)\vcentcolon=\left(\gamma_r-\frac{c}{p(N)},\gamma_r+\frac{c}{p(N)}\right)
\end{align*} for $N\in\mathcal{N}_2$. Now as $a(t)$ is a fixed sine polynomial independent of $N$ that satisfies $a(\gamma_r)\neq 0$ by \eqref{aCamplify}, we can choose a sufficiently small constant $c>0$, also independent of $N$, so that $\inf_{t\in (\gamma_r-c,\gamma_r+c)}|a(t)|\gg 1$ and hence we certainly have that $\inf_{t\in I_N(c)}|a(t)|\gg 1$ for all $N\in\mathcal{N}_2$ because $p(N)\geqslant 1$ since it is a positive integer by definition \eqref{pdefi}. By Lemma \ref{s_jsmallneargamma}, we can choose the constant $c>0$ sufficiently small so the functions $s_{j,\Lambda}^{(N)}$ will be arbitrarily small on the interval $I_N(c)$ and hence we may choose such a $c>0$, independent of $N$, so that all three of the functions $a(t), a(t)+s_{1,\Lambda}^{(N)}(t)$ and $a(t)-s_{2,\Lambda}^{(N)}(t)$ have the same sign for $t\in I_N(c)$, for each $N\in\mathcal{N}_2$. Hence, to prove the desired inequality \eqref{finaltaskinproof}, it suffices to show that there exists a $\kappa>0$, independent of $N$, such that the function $s_\Lambda^{(N)}(t)=s_{1,\Lambda}^{(N)}(t)+s_{2,\Lambda}^{(N)}(t)$ that we defined in \eqref{s_Lambda} takes the value $\kappa \sgn(a(\gamma_r))$ in the interval $I_N(c)$, as this would show precisely that 
\begin{equation*}
    |a(t)+s_{1,\Lambda}^{(N)}(t)|\geqslant |a(t)-s_{2,\Lambda}^{(N)}(t)|+\kappa
\end{equation*} for some $t=t(N)\in I_(c)$ since we just noted that $a, a+s_{1,\Lambda}^{(N)}$ and $a-s_{2,\Lambda}^{(N)}$ have the same sign on $I_N(c)$. From \eqref{lemsmallpart2} in Lemma \ref{s_jsmallneargamma}, we would then obtain the desired result that $|a(t)+s_{1}^{(N)}(t)|\geqslant |a(t)-s_{2}^{(N)}(t)|+\kappa/2$ for all large $N\in\mathcal{N}_2$. We state this final task in the proof of Proposition \ref{prop2} as a lemma.
\begin{lemma}
    Let $s_{\Lambda}^{(N)}$ be the function defined in \eqref{s_Lambda} for $N\in\mathcal{N}_2$, let $r=r(N)$ satisfy both of the bounds in \eqref{aCamplify}, and let $c>0$ be a fixed constant. Then there exists a fixed constant $\kappa>0$, independent of $N$, such that the following holds. For each sufficiently large $N\in\mathcal{N}_2$, there is a $t=t(N)\in(\gamma_r-c/p(N),\gamma_r+c/p(N))$ such that $\sgn(a(\gamma_r))\times s_{\Lambda}^{(N)}(t) \geqslant \kappa $.
\label{concludelemma}
\end{lemma}
\begin{proof}
We prove this under the assumption that $a(\gamma_r)>0$ as the case where $a(\gamma_r)<0$ is analogous. We also have that $\Im(C_j^{(N)})''(\gamma_r)>0$ as we noted after \eqref{aCamplify} that $a(\gamma_r)$ and $\Im(C_{j}^{(N)})''(\gamma_r)$ have the same sign. Note that by the definition \eqref{C_jdefi} in Lemma \ref{s_lambdadecolem}, the trigonometric polynomials $C_j^{(N)}$ with $N\in\mathcal{N}_2$ each consist of a finite number of terms which vanish at $\gamma_r=\frac{2\pi r}{D}$. We may therefore use Taylor's theorem to write for all $u\in(-c,c)$: 
\begin{align}
    \Im(C_j^{(N)})\left(\gamma_r+\frac{u}{p}\right)&=\frac{\Im(C_j)'(\gamma_r)}{p}u+\frac{\Im(C_j)''(\gamma_r)}{2p^2}u^2+O\left(u^3\right),\label{taylor}
\end{align}
where we used the trivial bound for the third derivative $\Im(C_j^{(N)})'''=O(p^3)$ that holds as $C_j^{(N)}$ is a finite sum of exponential whose frequencies are $O(p)$ by item (ii) in Lemma \ref{s_lambdadecolem}. Hence if we plug in the bound $\Im(C_j^{(N)})''(\gamma_r)\gg p^2$ that we obtained in \eqref{aCamplify}, then we deduce that there exists a $\kappa\gg_c 1$, independent of $N$, such that the function $\Im(C_j^{(N)})(\gamma_r+u/p)$ takes the value $\kappa$ for some $u\in(-c,c)$. In fact, we have proven the slightly stronger claim that for each $N\in\mathcal{N}_2$, there exists an interval $J_j=J_j(N)\subset (-\gamma_r-\frac{c}{p(N)},\gamma_r+\frac{c}{p(N)})$ of length $|J_j|\gg p(N)^{-1}$ such that $\Im(C_j^{(N)}(t))>\kappa$ for all $t\in J_j$. Let $J\subset (-\gamma_r-\frac{c}{p},\gamma_r+\frac{c}{p})$ be such an interval for $j=1$, so $|J|\gg p(N)^{-1}$ and we have for $t\in J$ that
\begin{equation}
    \Im(C_1^{(N)}(t))>\kappa.
\label{C_1large}
\end{equation}
We are now in a position to finish the argument, but we need to consider two cases based on whether $n_1(N)=O(p(N))$ or $p(N)=o(n_1(N))$ as $N\to\infty,N\in\mathcal{N}_2$, where we remind the reader that the $n_j(N)$ are integers such that the decomposition \eqref{s_lambdadecomposition} of $s^{(N)}_\Lambda$ holds. First, if $n_1(N)= O(p(N))$ then we note that we could have simply defined $n_1(N)=0$ and $C^{(N)}_1(t)=\sum_{m\in\Lambda^{(N,1)}}\delta_m\left(e^{imt}-e^{i(m-p^{(N)}_mD)t}\right)$ in the decomposition \eqref{s_lambdadecomposition} of $s_\Lambda^{(N)}$ and we check that all of the desired properties in items (i),(ii) and (iii) of Lemma \ref{s_lambdadecolem} still hold: for item (i) note that $\Lambda^{(N,1)}=\{m\in\Lambda:|m-n_1(N)|=O(p)\}=\{m\in\Lambda:|m|=O(p)\}$ as we are assuming that $n_1(N)=O(p)$, item (ii) is clear as the function $C^{(N)}_1$ above has degree $O(p)$ since we just noted that $m=O(p)$ for all $m\in\Lambda^{(N,1)}$, and item (iii) is trivial as could only have decreased $n_1(N)$. Hence, integrating the expression \eqref{s_lambdadecomposition} over $J=J(N)$ gives 
\begin{align*}
    \int_J s_\Lambda^{(N)}(t)\,dt &= \int_J \Im(C_1^{(N)}(t))\,dt+\sum_{j=2}^{\lambda'} \int_J\Im(C_j^{(N)}(t)e^{in_j(N)t})\,dt\\
    &= \int_J \Im(C_1^{(N)}(t))\,dt+O\left(\sum_{j=2}^{\lambda'} (n_j(N)+O(p))^{-1}\right)
\end{align*} where the bound in the second line follows from items (i) and (ii) in Lemma \ref{s_lambdadecolem} which show that the frequency of any term in $C_j^{(N)}(t)e^{in_j(N)t}$ is of size $n_j(N)+O(p)$. Recall now \eqref{n_jscales} which certainly implies that $p(N)=o(n_j(N))$ for $j\geqslant 2$ so that $n_j(N)^{-1}=o(p^{-1})$. Hence, as $J$ was chosen so that $|J|\gg p^{-1}$ and \eqref{C_1large} holds, the integral above is $>\kappa|J|+o(p(N)^{-1})\gg \kappa |J|$ as $N\to\infty,N\in\mathcal{N}_2$. This shows that for $N\in\mathcal{N}_2$ sufficiently large, there exists a $t\in J\subset(-\gamma_r-\frac{c}{p},\gamma_r+\frac{c}{p})$ such that $s_\Lambda(t)^{(N)}\gg \kappa$ which is the desired conclusion of Lemma \ref{concludelemma}.

\medskip

Finally, in the second case we may assume that $\frac{n_1(N)}{p(N)}\to \infty$ as $N\to \infty,N\in\mathcal{N}_2$. As we showed that there exists an interval $J$ satisfying \eqref{C_1large}, there certainly exists a $\gamma\in J\subset (-\gamma_r-\frac{c}{p},\gamma_r+\frac{c}{p})$ such that 
\begin{equation}
    |\Im(C_1^{(N)}(\gamma))|> \kappa. 
\label{gammalarge}
\end{equation}
Again recall that $p(N)=o(n_{j+1}(N)-n_{j}(N))$ for all $1\leqslant j\leqslant \lambda'-1$ by \eqref{n_jscales}, and as we are assuming that $p(N)=o(n_1(N))$ now, it is possible to define a variable $\omega=\omega(N)$ such that $p(N)=o(\omega)$ and 
\begin{align}
\omega(N) = o\big(\min\{n_1(N),n_2(N)-n_1(N),\dots,n_\lambda(N)-n_{\lambda-1}(N)\}\big).
\label{omegadefi}
\end{align}
As the trigonometric polynomials $C_j^{(N)}$ consist of $O(1)$ terms of degree $O(p(N))$ by item (ii) in Lemma \ref{s_lambdadecolem}, it follows that $\left|C_j^{(N)}(t)-C_j^{(N)}(\gamma)\right|\ll p(N)/\omega(N)=o_{N\to\infty}(1)$ holds uniformly for all $t=\gamma+\frac{u}{\omega(N)}$ where $u$ ranges over $(-1,1)$. Hence, for $u\in(-1,1)$ and $t=\gamma+\frac{u}{\omega(N)}$ we can rewrite \eqref{s_lambdadecomposition} as
\begin{equation}
    s_\Lambda^{(N)}(t)=\Im\left(\sum_{j=1}^{\lambda'} C_j^{(N)}(\gamma)e^{in_j t}\right)+o_{N\to\infty}(1)
\label{oscillatedecompo}
\end{equation}
where $o_{N\to\infty}(1)$ denotes a function that tends to $0$ as $N\to\infty$ uniformly for all $u\in(-1,1)$. Note now that the $C_j^{(N)}(\gamma)$ are fixed coefficients independent of $t$. Hence, if we integrate this expression over the interval $I=I(N)\vcentcolon=(\gamma-\frac{1}{\omega(N)},\gamma+\frac{1}{\omega(N)})$, we get
\begin{align}
    \int_I s_\Lambda^{(N)}(t)\,dt &= O\left(\sum_{j=1}^{\lambda'} n_j^{-1}\right)+o(|I|)=o(\omega^{-1})+o(|I|)=o(|I|)
\label{firstmoment}
\end{align}
as $N\to\infty$, and a similar calculation yields
\begin{align}
    \int_I \left(s_\Lambda^{(N)}(t)\right)^2\,dt &= |I|\left(\frac{1}{2}\sum_{j=1}^{\lambda'} |C_j^{(N)}(\gamma)|^2\right)+O(\max_{j>j'}(n_j-n_{j'})^{-1})+o(|I|)\nonumber\\
    &\gg |I|
\label{secondmoment}
\end{align}
as $N\to\infty$, where we used that $|I|=2\omega^{-1}$ so that $\max_{j>j'}(n_j-n_{j'})^{-1})=o(\omega^{-1})=o(|I|)$ by the definition \eqref{omegadefi} of $\omega$, and for the final inequality that $|\Im(C_1^{(N)}(\gamma))|> \kappa$ as we chose $\gamma$ to satisfy \eqref{gammalarge}. The following straightforward second moment then shows that there exists a $\kappa'\gg1$, independent of $N$, such that $s_\Lambda^{(N)}$ takes both the values $\pm\kappa'$ in $I(N)\subset(\gamma_r-\frac{c}{p(N)},\gamma_r+\frac{c}{p(N)})$ when $N\in\mathcal{N}_2$ is sufficiently large, which implies the desired result of Lemma \ref{concludelemma}. Indeed, note that if such a $\kappa'$ did not exist, then we could find functions $g^{(N)}(t)$ satisfying $\sup_I|g^{(N)}(t)|=o_{N\to\infty}(1)$ such that $f^{(N)}(t)\vcentcolon=\epsilon s_\Lambda^{(N)}(t)+g^{(N)}(t)$ is non-negative for all $t\in I$, for some $\epsilon\in\{\pm 1\}$. Then, combined with the fact that $\sup_I|g^{(N)}(t)|=o_{N\to\infty}(1)$ by definition, \eqref{firstmoment} shows that $\int_I f^{(N)} = o(|I|)$ while \eqref{secondmoment} gives $\int_I\left(f^{(N)}\right)^2\gg |I|$. However, note that the $f^{(N)}$ are uniformly bounded on $I$ for $N\in\mathcal{N}_2$ because the $g^{(N)}$ are uniformly bounded by definition and because $s_\Lambda^{(N)}$ is a trigonometric polynomial with $\lambda'=O(1)$ terms by \eqref{oscillatedecompo} and all these terms are $O(1)$ by item (ii) in Lemma \ref{s_lambdadecolem}. As $f^{(N)}$ is also non-negative by definition, we deduce that $\int_I\left(f^{(N)}\right)^2\ll \int_If^{(N)}$ as $N\to\infty,N\in\mathcal{N}_2$. This contradicts our estimates above and therefore finishes the proof of Lemma \ref{concludelemma};
    
\end{proof}
This completes the proof of Proposition \ref{prop2}.
\end{proof}

\bibliographystyle{plain}
\bibliography{referencesLittlewoodcosine}
\bigskip

\noindent
{\sc Mathematical Institute, Andrew Wiles Building, University of Oxford, Radcliffe
Observatory Quarter, Woodstock Road, Oxford, OX2 6GG, UK.}\newline
\href{mailto:benjamin.bedert@maths.ox.ac.uk}{\small benjamin.bedert@maths.ox.ac.uk}
\end{document}